\documentclass[12pt]{article}
\usepackage{e-jc}

\usepackage{amsmath,amssymb}
\usepackage{graphicx}
\usepackage[colorlinks=true,urlcolor=blue]{hyperref}

\usepackage{subfigure}

\usepackage[small,bf,hang]{caption} % Improved caption

\usepackage{booktabs}

\usepackage{multirow}

\newcommand{\inv}{^{-1}}

\newcommand{\Aut}{{\rm Aut}}

\graphicspath{{figures/}}

\dateline{XX}{XX}{XX}

\MSC{05C15, 05C21, 05C30, 05C40, 05C75, 68R10}

\Copyright{The authors. Released under the CC BY-ND license (International 4.0).}
\title{The smallest nontrivial snarks of oddness 4}

\author{Jan Goedgebeur\thanks{Supported by a Postdoctoral Fellowship
of the Research Foundation Flanders (FWO).}\\
\small Department of Applied Mathematics, Computer Science and Statistics\\[-0.8ex]
\small Ghent University\\[-0.8ex]
\small 9000 Ghent, Belgium\\
\small Computer Science Department\\[-0.8ex]
\small University of Mons\\[-0.8ex]
\small 7000 Mons, Belgium\\[-0.8ex]
\small\tt jan.goedgebeur@ugent.be\\
\\
Edita M\' a\v cajov\' a\thanks{Partially supported by VEGA 1/0876/16, VEGA 1/0813/18, and by
APVV-15-0220.}\\
\small Department of Computer Science\\[-0.8ex]
\small Comenius University\\[-0.8ex]
\small 842 48 Bratislava, Slovakia\\[-0.8ex]
\small\tt macajova@dcs.fmph.uniba.sk\\
\\
Martin \v Skoviera$^{\dag}$\\
\small Department of Computer Science\\[-0.8ex]
\small Comenius University\\[-0.8ex]
\small 842 48 Bratislava, Slovakia\\[-0.8ex]
\small\tt skoviera@dcs.fmph.uniba.sk\\
\\
}

\begin{document}

\maketitle

% E-JC papers must include an abstract. The abstract should consist of a
% succinct statement of background followed by a listing of the
% principal new results that are to be found in the paper. The abstract
% should be informative, clear, and as complete as possible. Phrases
% like "we investigate..." or "we study..." should be kept to a minimum
% in favor of "we prove that..."  or "we show that...".  Do not
% include equation numbers, unexpanded citations (such as "[23]"), or
% any other references to things in the paper that are not defined in
% the abstract. The abstract may be distributed without the rest of the
% paper so it must be entirely self-contained.  Try to include all words
% and phrases that someone might search for when looking for your paper.

\begin{abstract}
The oddness of a cubic graph is the smallest number of odd
circuits in a 2-factor of the graph. This invariant is widely
considered to be one of the most important measures of
uncolourability of cubic graphs and as such has been repeatedly
reoccurring in numerous investigations of problems and
conjectures surrounding snarks (connected cubic graphs
admitting no proper $3$-edge-colouring). In [Ars Math.\
Contemp.\ 16 (2019), 277--298] we have proved that the smallest
number of vertices of a snark with cyclic connectivity 4 and
oddness 4 is 44. We now show that there are exactly 31 such
snarks, all of them having girth 5. These snarks are built up
from subgraphs of the Petersen graph and a small number of
additional vertices. Depending on their structure they fall
into six classes, each class giving rise to an infinite family
of snarks with oddness at least $4$ with increasing order. We
explain the reasons why these snarks have oddness 4 and prove
that the 31 snarks form the complete set of snarks with cyclic
connectivity 4 and oddness 4 on 44 vertices. The proof is a
combination of a purely theoretical approach with extensive
computations performed by a computer.
\end{abstract}

%---------------------------------------------------------------

\section{Introduction}

This paper is a sequel of~\cite{GMS1} where we have proved that
the smallest number of vertices of a snark -- a connected cubic
graph whose edges cannot be properly coloured with three
colours -- which has cyclic connectivity $4$ and oddness at
least $4$ is $44$. The purpose of the present paper is to show
that there are precisely 31 such snarks, all of them having
oddness exactly~$4$, resistance $3$, and girth $5$. Together
with~\cite{GMS1}, this paper provides a partial answer to the
following question posed in \cite[Problem~2]{BrGHM}, leaving
open the existence of cyclically $5$-edge-connected snarks of
oddness at least $4$ on fewer than 44 vertices:
\medskip

\noindent\textbf{Problem~\cite{BrGHM}.} Which is the smallest
snark (with cyclic connectivity $\ge 4$ and girth $\ge 5$) of
oddness strictly greater than 2?

\medskip

The \textit{oddness} of a bridgeless cubic graph $G$ is the
smallest number of odd circuits in a $2$-factor of $G$, and the
\textit{resistance} of $G$ is the smallest number of vertices
(or edges) of $G$ whose removal yields a $3$-edge-colourable
graph. Both invariants are important measures of
uncolourability of cubic graphs and have been investigated by
numerous authors~\cite{Al, FMS-survey, HG, Hagglund, HK, LiSt,
Steffen:meas}. One of the reasons why these invariants have
recently received so much attention resides in the fact that
snarks with large resistance or oddness may provide potential
counterexamples to several profound conjectures such as the
cycle double cover conjecture, the 5-flow conjecture, and
others~\cite{HG, HK, J}.

The set $\mathcal{M}$ of all $31$ snarks of order 44, cyclic
connectivity $4$, and oddness $4$ has been constructed with the
help of a computer. A detailed description of its members
appears in Section~\ref{sec:31} where we also give a
computer-free proof that each of them has oddness at least $4$.
The equality can be established easily by specifying a
$2$-factor containing four odd circuits, which can be checked
directly.

The snarks constituting the set $\mathcal{M}$ fall into six
classes depending on their structure. A remarkable feature of
the set is that all of them are built up from subgraphs of the
Petersen graph and a small number of additional vertices.
The verification that $\mathcal{M}$ is a complete set of
snarks of order 44, cyclic connectivity~$4$, and oddness at
least $4$ combines purely mathematical considerations with
extensive computations; the proof can be found in
Section~\ref{sec:complete}. Its mathematical part is
essentially a mixture of edge-colouring and cyclic connectivity
arguments. In the computational part we perform an operation
that takes two cyclically $4$-edge-connected snarks $G_1$ and
$G_2$ of order at most 36 and creates from them -- in all
possible ways -- a cubic graph by either removing two adjacent
vertices or two nonadjacent edges, and connecting the resulting
$2$-valent vertices in $G_1$ with those in $G_2$. In total,
more than $2 \times 10^{13}$ graphs have been constructed and
checked for oddness~4. The entire computational effort
for this project amounts to 25 CPU years.

In Section~\ref{sect:properties} we analyse a sample of
887\,152 cyclically $4$-edge-connected snarks of oddness $4$
whose orders range from 46 to 52, as well as 872 snarks of
oddness $4$ with lower connectivity constructed in~\cite{G}.
We evaluate various invariants for them such as resistance,
perfect matching index, circumference, and others. The purpose
of this investigation is to provide grounds for possible
prediction of certain properties that snarks with higher
oddness might have in general.

We conclude this paper with several open problems. At the
end we append the adjacency lists of all 31 snarks constituting
the set $\mathcal{M}$.

\section{Preliminaries}\label{sec:prelim}
This section collects the most basic definitions and notation
needed for understanding the present paper. For a more detailed
introduction to the topic we refer the reader to our preceding
paper~\cite{GMS1}.

\medskip\noindent\textbf{2.1. Graphs and multipoles.}
All graphs in this paper are finite and for the most part
simple. However, for the sake of completeness we have to permit
graphs containing multiple edges or loops, although these
features will be usually excluded by the imposed connectivity
or colouring restrictions. For a graph $G$ and a subgraph
$H\subseteq G$ we let $|G|$ denote the number of vertices of
$G$, and $G[H]$ the subgraph of $G$ induced by the vertex set
of $H$.

Throughout this paper we use multipoles as a convenient tool
for constructing graphs. Every edge of a multipole has two ends
and each end can, but need not, be incident with a vertex. An
edge which has one end incident with a vertex and the other not
is called a \textit{dangling edge}, and if neither end of an
edge is incident with a vertex, it is called an
\textit{isolated edge}. An end of an edge that is not incident
with a vertex is called a \textit{semiedge}. A multipole with
$k$ semiedges is called a \textit{$k$-pole}. Two semiedges $s$
and $t$ of a multipole can be joined to produce an edge $s*t$
connecting the end-vertices of the corresponding dangling
edges. Given two $k$-poles $M$ and $N$ with semiedges $s_1,
\ldots, s_k$ and $t_1,\ldots, t_k$, respectively, we define
their \textit{complete junction} $M*N$ to be the graph obtained
by performing the junctions $s_i*t_i$ for each $i\in\{1,\ldots,
k\}$. A \textit{partial junction} is defined in a similar way
except that a proper subset of semiedges of $M$ is joined to
semiedges of $N$. Partial junctions can be used to construct
larger multipoles from smaller ones. In either case, whenever a
junction of two multipoles is to be performed, we assume that
their semiedges are assigned a fixed linear order.

Semiedges in multipoles are often grouped into pairwise
disjoint sets, called \textit{connectors}.  The \textit{size}
of a connector is the number of its semiedges. A connector of
size $n$ is often referred to as an \textit{$n$-connector}. An
$(n_1, n_2,\ldots, n_k)$-pole is a multipole with
$n_1+n_2+\cdots + n_k$ semiedges which are distributed into $k$
connectors $S_1,S_2,\ldots, S_k$ such that the connector $S_i$
is of size $n_i$. A multipole with two connectors is also
called a \textit{dipole}.

Let $D'(S_1',S_2')$ and $D''(S_1'',S_2'')$ be two dipoles with
connectors $S_1'$, $S_2'$ and $S_1''$, $S_2''$, respectively.
If $|S_1'|=|S_2|$, and each of these two connectors is endowed
with a linear order, we can construct a new dipole
$$D(S_1,S_2)=D'(S_1',S_2')\circ D''(S_1'',S_2'')$$
with connectors $S_1=S_1'$ and $S_2=S_2''$ by performing the
junctions of the semiedges from $S_2'$ with those $S_1''$ with
respect to the corresponding orderings. The resulting dipole is
called the \textit{junction} of $D'(S_1',S_2')$ and
$D''(S_1'',S_2'')$.

\medskip\noindent\textbf{2.2. Cyclic connectivity.}
Let $G$ be a connected graph. An \textit{edge-cut} of a
graph~$G$ is any set $S$ of edges of $G$ such that $G-S$ is
disconnected. For example, if $H$ a proper subset of vertices
or induced subgraph of $G$, then the set $\delta_G(H)$ of all
edges with exactly one end in $H$ is an edge-cut in $G$. An
edge-cut $S$ is said to be \textit{cycle-separating} if at
least two components of $G-S$ contain cycles. We say that a
connected graph $G$ is \textit{cyclically $k$-edge-connected}
if no set of fewer than $k$ edges is cycle-separating in $G$.
Let $\beta(G)=|E(G)|-|V(G)|+1$ denote the cycle rank of $G$.
The \textit{cyclic connectivity} of $G$, denoted by $\zeta(G)$,
is the largest number $k\le\beta(G)$ for which $G$ is
cyclically $k$-connected (cf.~\cite{NS:cc,R}). It is not
difficult to see that $\zeta(G)=\beta(G)$ if and only if $G$
any two curcuits of $G$ have a vertex in common. For cubic
graphs this can happen only when $G$ is the complete bipartite
graph $K_{3,3}$, the complete graph $K_4$ on four vertices, or
the graph consisting of two vertices and three parallel edges
joining them.

For a cubic graph $G$ with $\zeta(G)\le 3$, the value
$\zeta(G)$ coincides with the usual vertex-connectivity and
edge-connectivity of $G$, so cyclic connectivity provides a
natural extension of the classical connectivity parameters for
cubic graphs.  Another useful observation is that the value of
cyclic connectivity remains invariant under subdivisions and
adjoining new vertices of degree~$1$. In particular,
homeomorphic graphs have the same value of cyclic connectivity.

For edge-cuts that separate an acyclic component from the rest
of the graph we have the following easy but useful observation.

\begin{lemma}\label{lemma:acyclic}
A connected acyclic $k$-pole has $k-2$ vertices.
\end{lemma}

\medskip\noindent\textbf{2.3. Edge-colourings.}
A \textit{$3$-edge-colouring} of a graph $G$ is a mapping
$\varphi\colon E(G)\to\{1,2,3\}$ such that adjacent edges
receive distinct colours; the same definition applies to
multipoles. A graph or a multipole which admits a
$3$-edge-colouring will be called \textit{colourable},
otherwise it will be called \textit{uncolourable}. A
$2$-connected uncolourable cubic graph is called a
\textit{snark}. A snark is \textit{nontrivial} if it is
cyclically $4$-edge-connected and has girth at least $5$.

In the study of colourings of cubic graphs it is often
convenient to take the colours to be the non-zero elements of
the group $\mathbb{Z}_2\times\mathbb{Z}_2$, because in this
case 3-edge-colourings correspond to nowhere-zero
$\mathbb{Z}_2\times\mathbb{Z}_2$-flows. We identify the colours
$(0,1)$, $(1,0)$, and $(1,1)$ with $1$, $2$, and $3$,
respectively.

The following well known lemma -- in fact, an immediate
consequence of flow continuity -- is a fundamental tool in the
study of snarks.

\begin{theorem}\label{lemma:parity} {\rm (Parity Lemma)}
Let $M$ be a $k$-pole endowed with a proper $3$-edge-colouring
with colours $1$, $2$, and $3$. If the set of all semiedges of
$M$ contains $k_i$ edges of colour $i$ for $i\in\{1,2,3\}$,
then
$$
k_1\equiv k_2\equiv k_3\equiv k\pmod 2.
$$
\end{theorem}

In this paper we study snarks that are far from being
$3$-edge-colourable. Two measures of uncolourability are
relevant for this paper. The \emph{oddness} $\omega(G)$ of a
bridgeless cubic graph $G$ is the smallest number of odd
circuits in a 2-factor of $G$. The \emph{resistance} $\rho(G)$
of a cubic graph $G$ is the smallest number of vertices of $G$
which have to be removed in order to obtain a colourable graph.
Somewhat surprisingly, the required number of vertices to be
deleted is the same as the number of edges that have to be
deleted in order to get a $3$-edge-colourable graph (see
\cite[Theorem~2.7]{Steffen:class}). In fact, in many cases it
is more convenient to delete edges rather than vertices.

Obviously, if $G$ is colourable, then $\omega(G)=\rho(G)=0$. If
$G$ is uncolourable, then both $\omega(G)\ge 2$ and $\rho(G)\ge
2$. Observe that for every bridgeless cubic graph $G$ we have
$\rho(G) \le \omega(G)$ since deleting one edge from each odd
circuit in a $2$-factor leaves a colourable graph. On the other
hand, the Parity Lemma implies that $\rho(G)$ never equals $1$,
which together with a standard Kempe chain recolouring argument
yields that $\rho(G)=2$ if and only if $\omega(G)=2$
\cite[Lemma~2.5]{Steffen:class}. The difference between
$\omega(G)$ and $\rho(G)$ can be arbitrarily large in
general~\cite{Steffen:meas}, nevertheless, resistance can serve
as a convenient lower bound for oddness because it is somewhat
easier to handle.

By the Parity Lemma, every colouring of a $4$-pole has one of
the following types: $1111$, $1122$, $1212$, and $1221$ (for a
precise definition of the type of a colouring see~\cite{GMS1}).
Observe that every colourable $4$-pole admits at least two
different types of colourings. Indeed, we can start with any
colouring and switch the colours along an arbitrary Kempe chain
to obtain a colouring of another type. Colourable $4$-poles
thus can have two, three, or four different types of
colourings. Those attaining exactly two types are particularly
important for the study of snarks; we call them
\textit{colour-open} $4$-poles, as opposed to
\textit{colour-closed} multipoles discussed in more detail
in~\cite{NS:dec}.

There are two types of colour-open $4$-poles. A $4$-pole $M$
will be called \textit{isochromatic} if its semiedges can be
partitioned into two pairs such that in every colouring of $M$
the semiedges within each pair receive the same colour. A
$4$-pole $M$ will be called \textit{heterochromatic} of its
semiedges can be partitioned into two pairs such that in every
colouring of $M$ the semiedges within each pair receive
distinct colours. Typical examples of isochromatic and
heterochromatic $4$-poles are depicted in Figure~\ref{fig:I} and
Figure~\ref{fig:H}, respectively.

The adjectives ``isochromatic" and ``heterochromatic" can be
similarly applied to graphs: a subgraph of a cubic graph will
be called \textit{isochromatic} if attaching a dangling edge to
every $2$-valent vertex produces an isochromatic $4$-pole; a
\textit{heterochromatic} subgraph is defined similarly.

\section{The 31 snarks}\label{sec:31}
Let $\mathcal{S}_{36}$ denote the set of all cyclically
4-edge-connected snarks with at most 36 vertices; as mentioned
in~\cite{GMS1}, the set $\mathcal{S}_{36}$ consists of
$432~105~682$ nonisomorphic graphs. For any two snarks $G_1$
and $G_2$ from $\mathcal{S}_{36}$ let us apply the following
operation:
\begin{itemize}
\item From each $G_i$ form a $4$-pole $M_i$ by either
    removing two adjacent vertices or two nonadjacent edges
    and by retaining the dangling edges.
\item Construct a cubic graph $G$ by identifying the
    semiedges of $M_1$ with those of $M_2$ after possibly
    applying a permutation to the  semiedges of $M_1$ or
    $M_2$.
\end{itemize}
The resulting graph $G$ will be called a \textit{$4$-join} of
$G_1$ and $G_2$. Define $\mathcal{M}$ to be the set all
pairwise nonisomorphic cyclically $4$-edge-connected snarks of
order $44$ with oddness at least $4$ that can be expressed as a
$4$-join of two not necessarily distinct graphs from
$\mathcal{S}_{36}$.

We have implemented a program which applies a $4$-join in all
possible ways to any two input graphs; see~\cite{GMS1} for more
details concerning the program. We have applied this program in
all possible ways to every pair of graphs from
$\mathcal{S}_{36}$ that lead to a graph on 44 vertices, and
then tested which of the constructed graphs have oddness at
least 4. This computation took approximately 2 CPU years on a
cluster consisting of Intel Xeon E5-2660 CPU's at 2.60GHz and
produced $31$ snarks with oddness exactly $4$.

\begin{observation} \label{obs:31_snarks}
The set $\mathcal{M}$ consists of exactly $31$ snarks, each of
them having oddness exactly $4$ and girth $5$.
\end{observation}

In the remainder of this section we describe the 31 snarks in
detail and provide a computer-free proof that each of them has
oddness at least $4$. Their adjacency lists, displayed in the
order as they were generated, can be found in Appendix. Graph
number 28 is illustrated in Figure~\ref{fig:graph28}, two more
of the 31 graphs are depicted in Figure~1 and Figure~4 in our
previous paper~\cite{GMS1} which represent graphs number 15 and
17, respectively.

\begin{figure}[htbp]
	\centering
	\includegraphics[width=0.4\textwidth]{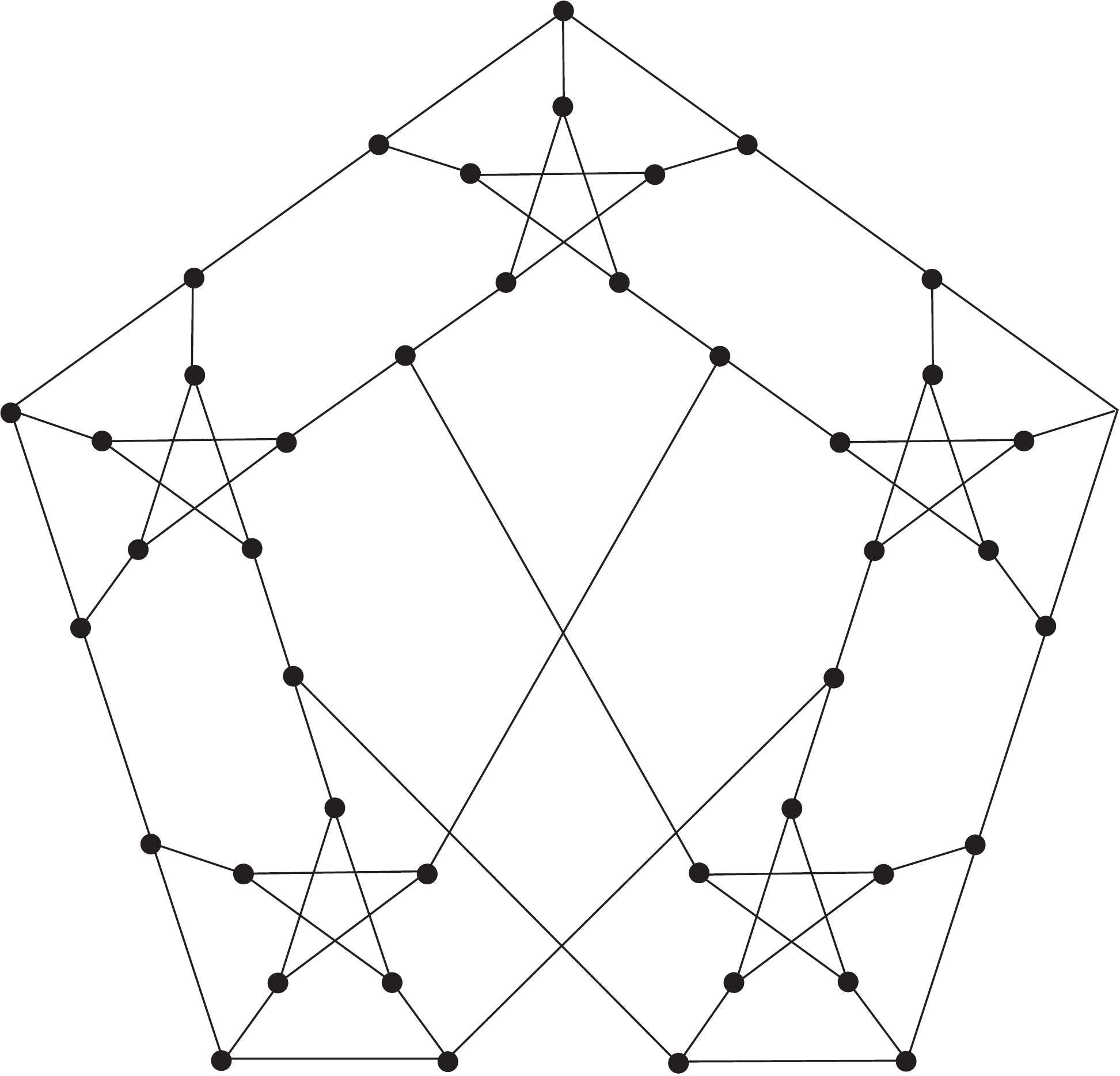}
  \caption{Graph number 28 of the 31 graphs in $\mathcal{M}$}
  \label{fig:graph28}
\end{figure}

\subsection*{Building blocks}
The basic building blocks of all snarks constituting the set
$\mathcal{M}$ are five multipoles $\mathbf{I}$, $\mathbf{H}_1$,
$\mathbf{H}_2$, $\mathbf{T}$, and $\mathbf{N}$, described
below, all of them arising from the Petersen graph by removing
vertices or severing edges. Every individual member of
$\mathcal{M}$ may have a small number of additional vertices
not belonging to any of these subgraphs.

\begin{figure}[htbp]
	\centering
	\includegraphics[width=0.4\textwidth]{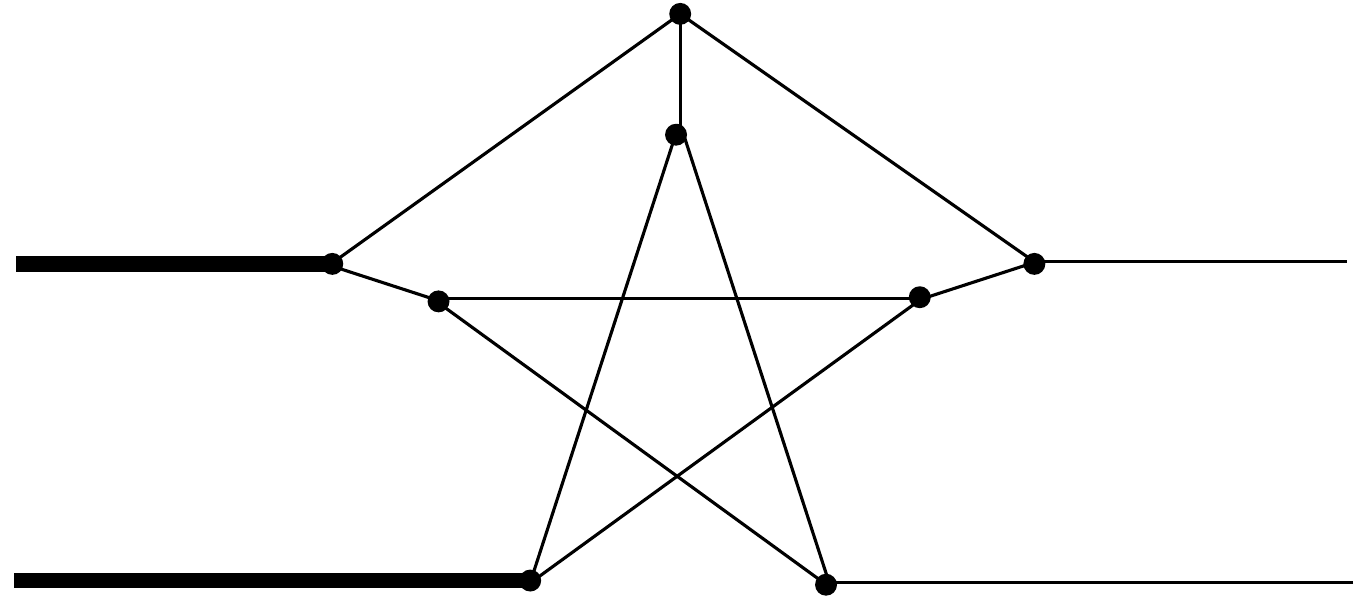}
  \caption{Building block $\mathbf{I}$}
  \label{fig:I}
\end{figure}

\begin{itemize}
\item[1.] Let $\mathbf{I}$ denote the $(2,2)$-pole arising
    from the Petersen graph by removing two adjacent
    vertices and grouping the semiedges formerly incident
    with the same vertex to the same connector; it is shown
    in Figure~\ref{fig:I}. Since every $3$-edge-colouring of
    $\mathbf{I}$ assigns the edges in the same connector
    the same colour, $\mathbf{I}$ is an is an isochromatic
    $4$-pole. It is the only isochromatic $4$-pole on eight
    vertices and at the same time the smallest connected
    isochromatic $4$-pole. In the symbolic representation
    of $\mathbf{I}$ we represent the edges in one of the
    connectors by bold lines, see Figure~\ref{fig:basic}.
    This clearly determines the other connector as well.
    Due to the symmetry of the Petersen graph, the two
    connectors of $\mathbf{I}$ are interchangeable.

    We emphasise that the connectors of $\mathbf{I}$, as
    well as those of the other four building blocks
    $\mathbf{H}_1$, $\mathbf{H}_2$, $\mathbf{T}$, and
    $\mathbf{N}$, are unordered. Different orderings are
    needed for the construction of the members of
    $\mathcal{M}$, a fact which partially explains a
    relatively large size of this set.

\begin{figure}[htbp]
	\centering
	\includegraphics[width=0.75\textwidth]{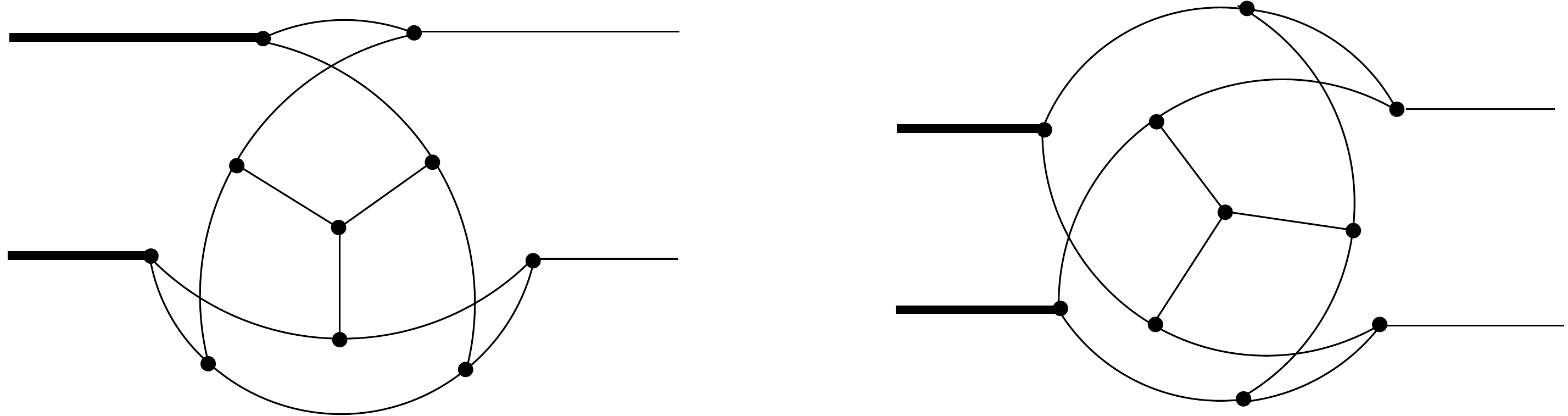}
  \caption{Building blocks $\mathbf{H}_1$ and $\mathbf{H}_2$}
  \label{fig:H}
\end{figure}

\item[2.] Let $\mathbf{H}$ denote a $(2,2)$-pole formed
    from the Petersen graph by severing two independent
    edges and grouping the semiedges arising from the same
    edge to the same connector. Every $3$-edge-colouring of
    $\mathbf{H}$ assigns different colours to the dangling
    edges within the same connector, so $\mathbf{H}$ is a
    heterochromatic $4$-pole. There are two ways how to
    select a pair of independent edges in the Petersen
    graph -- either at distance $1$ or at distance $2$.
    Accordingly, there exist two nonisomorphic
    heterochromatic $4$-poles on ten vertices, denoted by
    $\mathbf{H}_1$ and $\mathbf{H}_2$, respectively.
    Observe that there exists no bridgeless heterochromatic
    $4$-pole with fewer vertices. The two heterochromatic
    $4$-poles on ten vertices are displayed in
    Figure~\ref{fig:H}. In the symbolic representation of
    $\mathbf{H}$ we distinguish the two connectors of
    $\mathbf{H}$ again by using bold lines for one of the
    connectors, see Figure~\ref{fig:basic}. Again, the
    connectors of both $\mathbf{H}_1$ and $\mathbf{H}_2$
    are interchangeable.

\begin{figure}[htbp]
	\centering
	\includegraphics[width=0.4\textwidth]{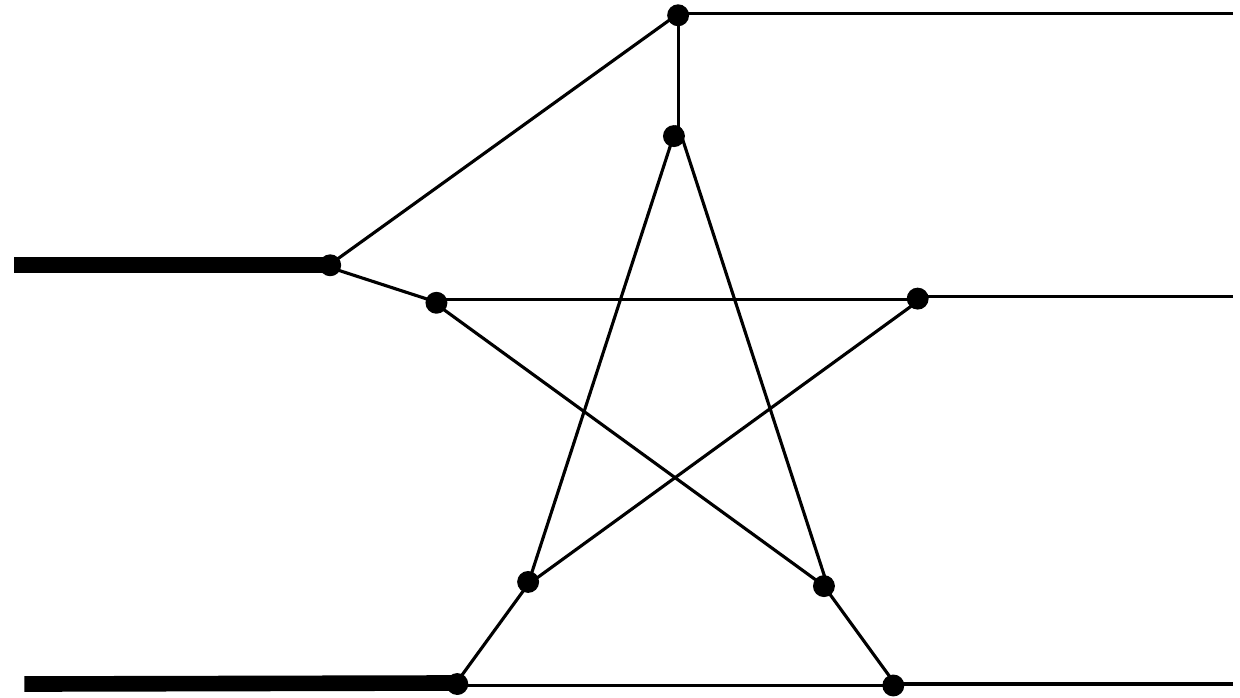}
  \caption{Building block $\mathbf{T}$}
  \label{fig:T}
\end{figure}

\item[3.] Let $\mathbf{T}$ denote the $(2,3)$-pole obtained
    from the Petersen graph by removing an arbitrary vertex
    $v$ and severing an edge $e$ not incident with $v$; the
    semiedges formerly incident with $v$ are put to one
    connector and those arising by severing $e$ are put to
    the other connector. The resulting $(2,3)$-pole is
    shown in Figure~\ref{fig:T}. Given an arbitrary
    $3$-edge-colouring of $\mathbf{T}$, the edges of the
    $2$-connector receive two distinct colours $x$ and $y$,
    while the edges of the $3$-connector receive the
    colours $x+y$, $z$, $z$, where $z$ is any colour from
    $\{1,2,3\}$. In the symbolic representation of
    $\mathbf{T}$ the edges of the $2$-connector are drawn
    bold, see Figure~\ref{fig:basic}.

\begin{figure}[htbp]
	\centering
	\includegraphics[width=0.4\textwidth]{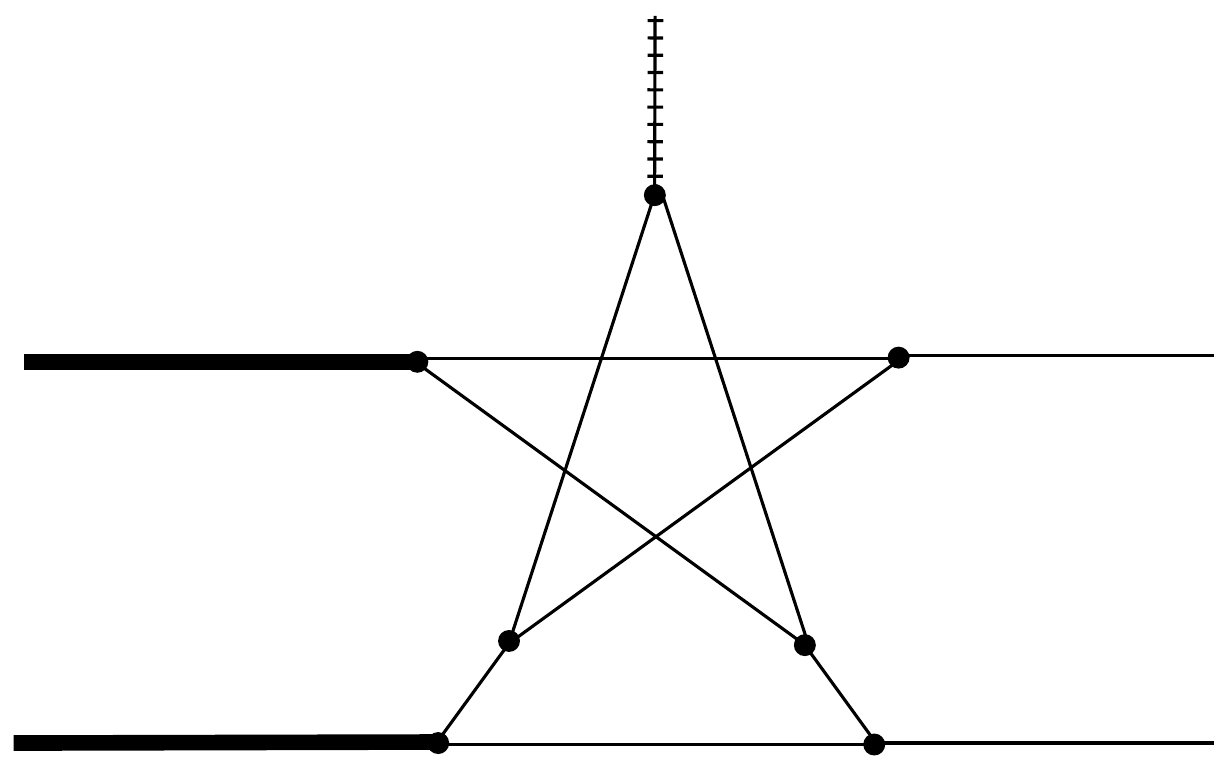}
  \caption{Building block $\mathbf{N}$}
  \label{fig:N}
\end{figure}

\item[4.] Let $\mathbf{N}$ denote the $(2,2,1)$-pole
    arising from the  Petersen graph by removing a path of
    length $2$. The two $2$-connectors consist of the edges
    formerly incident with the same end-vertex of the path,
    the $1$-connector gets the remaining edge. The
    resulting $(2,2,1)$-pole is shown in Figure~\ref{fig:N}.
    The important property of $\mathbf{N}$ consists in the
    fact that every $3$-edge-colouring of $\mathbf{N}$
    assigns the edges of one of the $2$-connectors two
    distinct colours $x$ and $y$ while the edges of the
    other $2$-connector receive the same colour
    $z\in\{1,2,3\}$; the fifth edge of $\mathbf{N}$ is
    coloured $x+y$. In the symbolic representation of
    $\mathbf{N}$ the edges of one $2$-connector are drawn
    bold and the $1$-connector edge is drawn dashed, see
    Figure~\ref{fig:basic}. As with $\mathbf{I}$ and
    $\mathbf{H}$ before, the two $2$-connectors of
    $\mathbf{N}$ are interchangeable.
\end{itemize}

\begin{figure}[htbp]
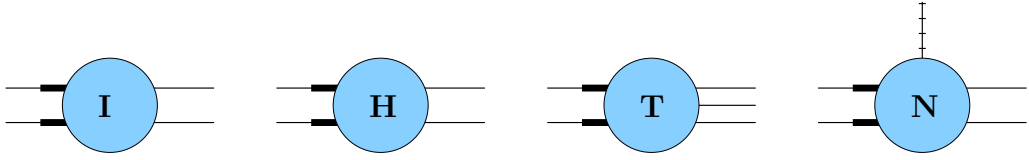
%
  \begin{center}
    \subfigure%[Multipole $\mathbf{I}$]
    {\scalebox{0.48}{\input{figures/dipI}}}
    \hspace{5mm}
    \subfigure%[Multipole $\mathbf{H}_i$]
    {\scalebox{0.48}{\input{figures/dipHi}}}
    \hspace{5mm}
    \subfigure%[Multipole $\mathbf{T}$]
    {\scalebox{0.48}{\input{figures/dipT}}}
    \hspace{5mm}
    \subfigure%[Multipole $\mathbf{N}$]
    {\scalebox{0.48}{\input{figures/dipN}}}\\[-10pt]
    \caption{Basic building blocks}%
    \label{fig:basic}%
  \end{center}
\end{figure}

\subsection*{Six classes}
We divide the 31 snarks of $\mathcal{M}$ into six classes
depending on the number of disjoint copies of $\mathbf{I}$,
$\mathbf{H}$, $\mathbf{T}$, and $\mathbf{N}$, and on the number
of additional vertices in the graph. For example, by
$2\mathbf{H}+2\mathbf{I}+\mathbf{N}+1$ we denote the set of all
snarks from $\mathcal{M}$ that consist of two copies of
$\mathbf{H}$, which need not be isomorphic, two copies of
$\mathbf{I}$, one copy of $\mathbf{N}$, and one additional
vertex. We do not distinguish between the two varieties
$\mathbf{H}_1$ and $\mathbf{H}_2$ of $\mathbf{H}$ because both
of them play the same structural role within the snark in
question and their contribution to increasing the oddness is
the same. The six classes of $\mathcal{M}$ are
$$
2\mathbf{H}+3\mathbf{I},\
2\mathbf{H}+2\mathbf{I}+\mathbf{N}+1,\
\mathbf{H}+4\mathbf{I}+2,\
\mathbf{H}+3\mathbf{I}+\mathbf{T}+1,\
5\mathbf{I}+4, \text{ and }\ 4\mathbf{I}+\mathbf{T}+3.
$$

In the rest of this section we describe each of the classes in
detail and prove that every member $G\in\mathcal{M}$ has
$\omega(G)\ge 4$. In our discussion we will be employing
certain standard combinations of the multipoles $\mathbf{I}$,
$\mathbf{H}$, and $\mathbf{T}$ defined by means of junctions.
The order of semiedges in connectors is in all cases
irrelevant. We define the $(2,2)$-poles
$\mathbf{Z}_1=\mathbf{I}\circ\mathbf{H}$ and
$\mathbf{Z}_2=\mathbf{I}\circ\mathbf{H}\circ\mathbf{I}$, the
$(2,3)$-pole $\mathbf{Z}_4=\mathbf{I}\circ\mathbf{T}$, where
the junction involves the $2$-connector of~$\mathbf{T}$, and
the $(2,2,1)$-pole $\mathbf{Z}_3$ defined as follows: in the
$(2,2)$-pole $\mathbf{I}\circ\mathbf{I}$ subdivide one of the
edges between the two copies of $\mathbf{I}$ and subsequently
attach a dangling edge to the new vertex of degree $2$; the
connectors of $\mathbf{Z}_3$ are defined in the obvious way.
The multipoles $\mathbf{Z}_1$, $\mathbf{Z}_2$, $\mathbf{Z}_3$,
and $\mathbf{Z}_4$ are illustrated in
Figure~\ref{fig:dipolesZ}.

\begin{figure}[htbp]
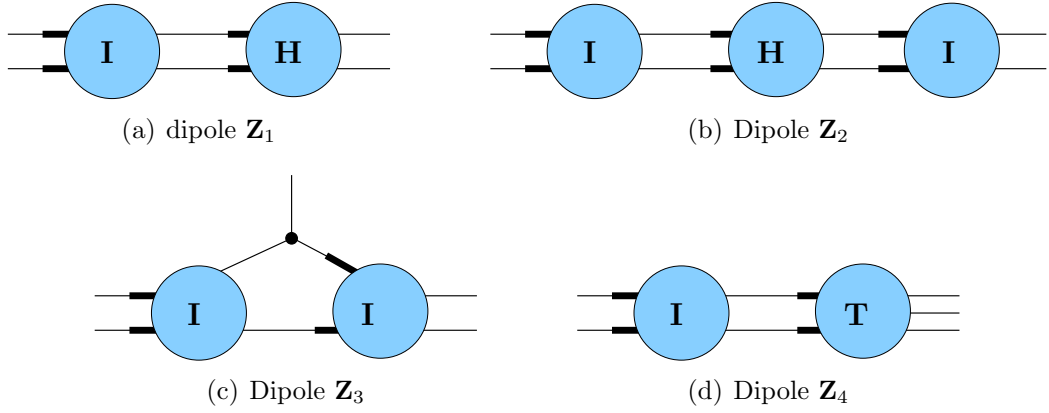
%
  \begin{center}
    \subfigure[dipole $\mathbf{Z}_1$]{\scalebox{0.48}{\input{figures/dipZ1}}}
    \hspace{1cm}
    \subfigure[Dipole $\mathbf{Z}_2$]{\scalebox{0.48}{\input{figures/dipZ2}}}\\
    \subfigure[Dipole $\mathbf{Z}_3$]{\scalebox{0.48}{\input{figures/dipZ3}}}
    \hspace{1cm}
    \subfigure[Dipole $\mathbf{Z}_4$]{\scalebox{0.48}{\input{figures/dipZ4}}}\\[-10pt]
    \caption{Dipoles $\mathbf{Z}_i$}%
    \label{fig:dipolesZ}%
  \end{center}
\end{figure}

\begin{lemma}\label{lemma:Z}
The following statements hold true:
$$\rho(\mathbf{Z}_1)=\rho(\mathbf{Z}_3)=\rho(\mathbf{Z}_4)=1\
\text { and } \ \rho(\mathbf{Z}_2)=2.$$
\end{lemma}

\begin{proof}
Observe that $\mathbf{Z}_1$ is uncolourable because it is a
junction of an isochromatic $(2,2)$-pole $\mathbf{I}$ with a
heterochromatic $(2,2)$-pole $\mathbf{H}$. Hence,
$\rho(\mathbf{Z}_1)\ge 1$. Similarly, every colouring of
$\mathbf{T}$ assigns its $2$-connector two distinct colours
while $\mathbf{I}$ is isochromatic, so
$\mathbf{Z}_4=\mathbf{I}\circ\mathbf{T}$ is uncolourable too,
and therefore $\rho(\mathbf{Z}_4)\ge 1$ again. If the
$(2,2,1)$-pole $\mathbf{Z}_3$ was colourable, then every
$3$-edge-colouring of $\mathbf{Z}_3$ would assign the edges
adjacent to the dangling edge of the $1$-connector two distinct
colours. At most one of these colours would match the colour of
the edge connecting the two copies of $\mathbf{I}$ in
$\mathbf{Z}_3$. Therefore, the isochromatic property of at
least one copy of $\mathbf{I}$ in $\mathbf{Z}_3$ would always
be violated. Hence, $\mathbf{Z}_3$ is uncolourable, and
therefore $\rho(\mathbf{Z}_3)\ge 1$.

Finally we prove that $\rho(\mathbf{Z}_2)\ge 2$. Since
$\mathbf{Z}_1$ is contained in $\mathbf{Z}_2$ and
$\rho(\mathbf{Z}_1)\ge 1$, we infer that $\rho(\mathbf{Z}_2)\ge
1$ as well. In particular, the graph $K$ obtained from
$\mathbf{Z}_2$ by identifying the semiedges within each
connector is a snark. To prove that $\rho(\mathbf{Z}_2)\ge 2$
suppose to the contrary that $\rho(\mathbf{Z}_2)=1$, and let
$v$ be a vertex such that $\mathbf{Z}_2-v$ is
$3$-edge-colourable. Clearly, $v$ cannot belong to a copy of
$\mathbf{I}$ for otherwise $\mathbf{Z}_2-v$ would still contain
a copy of $\mathbf{Z}_1$ and therefore would be uncolourable.
Thus $v$ must belong to the copy of $\mathbf{H}$, and hence
every $3$-edge-colouring of $\mathbf{Z}_2-v$ must assign the
same colour to both semiedges in any of the connectors. Now we
can match the semiedges of each connector, thereby obtaining a
$3$-edge-colouring of $K-v$. Since $K$ is a snark, such a
colouring does not exist. This contradiction proves that
$\rho(\mathbf{Z}_2)\ge 2$.

To establish the required equalities for each particular
multipole one has to display the corresponding colourings.
Finding such colourings is straightforward, and therefore is
left to the reader.
\end{proof}

In our analysis of the $31$ snarks we will often need to
distinguish between different copies of the same basic building
block $\mathbf{B}\in\{\mathbf{I}, \mathbf{H}, \mathbf{T},
\mathbf{N}\}$. For this purpose we will be using upper indices,
for example $\mathbf{B}^1$, $\mathbf{B}^2$, etc.

\subsection*{Class 1: \textbf{2H + 3I}\, (7 graphs)} %Group 1: 2xH, 3xI }

\begin{figure}[htbp]%
  \begin{center}
    \subfigure[Class~1a (3 graphs)]{\scalebox{0.4}{\input{figures/sk1a}}}
    \hspace{1cm}
    \subfigure[Class~1b (4 graphs)]{\scalebox{0.4}{\input{figures/sk1b}}}\\[-10pt]
    \caption{Class 1}%
    \label{fig:class1}%
  \end{center}
\end{figure}

This class splits into two subclasses, Class~1a and Class~1b,
depending on how the building blocks are connected between each
other. Both subclasses are illustrated in
Figure~\ref{fig:class1}. Each graph $G$ from Class~1 contains a
copy of $\mathbf{Z}_1$ and a copy of $\mathbf{Z}_2$. Since the
copies of  $\mathbf{Z}_1$ and $\mathbf{Z}_2$ are disjoint,
Lemma~\ref{lemma:Z} implies that $\rho(G)\ge 3$ and hence
$\omega(G)\ge 4$.

Class~1a consists of graphs 15, 17, and 18, while Class~1b
contains graphs 1, 4, 21, and 24.

\subsection*{Class 2: \textbf{2H + 2I + N + 1}\, (4 graphs)}
The structure of graphs from Class~2 is illustrated in
Figure~\ref{fig:class2}. It can be seen that every graph $G$ from
Class~2 contains two disjoint copies of $\mathbf{Z}_1$,
therefore $\rho(G)\ge 2$. Note each of the  induced subgraphs
$[\mathbf{H}^1\cup\mathbf{N}\cup\mathbf{H}^2]$,
$[\mathbf{H}^1\cup\mathbf{I}^1]$,
$[\mathbf{H}^2\cup\mathbf{I}^2]$, and
$[\mathbf{I}^1\cup\mathbf{I}^2\cup \{v\}]$ is uncolourable, the
last of them being isomorphic to $\mathbf{Z}_3$. Thus if there
exist two vertices $u$ and $w$ in $G$ such that $G-\{u,w\}$ is
colourable, then either one of them lies in $\mathbf{H}^1$ and
the other lies in $\mathbf{I}^{2}$, or one of them lies in
$\mathbf{H}^2$ and the other lies in $\mathbf{I}^{1}$. In all
other cases one of the mentioned uncolourable subgraphs will
remain intact. Without loss of generality we may assume that
$u$ belongs to $\mathbf{H}^1$ and $w$ belongs $\mathbf{I}^2$.
Since $\mathbf{I}^1$ is isochromatic, the edges between
$\mathbf{I}^1$ and $\mathbf{H}^1$ have the same colour. Because
$\mathbf{H}^2$ is heterochromatic, the colours the edges
between $\mathbf{H}^2$ and $\mathbf{N}$ are different. The
inverting property of $\mathbf{N}$ ensures that the edges
between $\mathbf{N}$ and $\mathbf{H}^1$ also have the same
colour. If we now identify the dangling edges within each
connector of $\mathbf{H}^1$, we obtain a colouring of the
Petersen graph with one vertex removed. This contradiction
proves that that $\rho(G)>2$, which means that $\omega(G)\ge
4$, as required.

Class~2 contains graphs 10, 11, 19, and 27.

\begin{figure}[htbp]
  \centerline{
   \scalebox{0.4}{
     \input{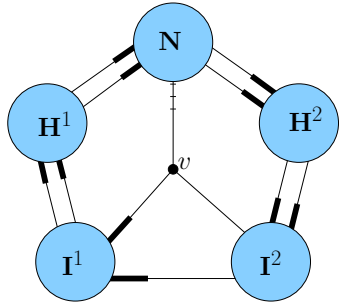}
   }
  }
   \caption{Class 2 (4 graphs)}
   \label{fig:class2}
\end{figure}

\subsection*{Class 3: \textbf{H + 4I + 2}\, (6 graphs) } %Group 2: 1xH, 4xI}

\begin{figure}[htbp]%
  \begin{center}
    \subfigure[Class 3a (2 graphs)]{\scalebox{0.37}{\input{figures/sk3ane}}}
    \hspace{1mm}
    \subfigure[Class 3b (2 graphs)]{\scalebox{0.37}{\input{figures/sk3bne}}}
%    \hspace{1mm}
    \\
    \subfigure[Class 3c (2 graphs)]{\scalebox{0.265}{\input{figures/sk3cne}}}
    \caption{Class 3}%
    \label{fig:class3}%
  \end{center}
\end{figure}

This class splits into three subclasses, 3a, 3b, and 3c. Their
structure is represented in Figure~\ref{fig:class3}. Every graph
$G$ from Class~3a and Class~3b contains a copy of
$\mathbf{Z}_2$ and a copy $\mathbf{Z}_3$, which are disjoint,
therefore $\rho(G)\ge 3$ and $\omega(G)\ge 4$. Class~3c is
somewhat different and requires a separate argument. Consider
an arbitrary graph $G$ graph from  Class~3c. Since $G$ contains
a copy of $\mathbf{Z}_1$, we have $\rho(G)\ge 2$. Note that
each of the subgraphs $[\mathbf{H}\cup\mathbf{I}^1]$,
$[\mathbf{I}^2\cup\mathbf{I}^3\cup\{x\}]$, and
$[\mathbf{I}^3\cup\mathbf{I}^4\cup\{y\}]$ has resistance 1.
Thus if there exist two vertices $u$ and $w$ in $G$ such that
$G-\{u,w\}$ is $3$-edge-colourable, then one of them, say $u$,
lies in $[\mathbf{H}\cup\mathbf{I}^1]$ and the other belongs to
$\mathbf{I}^3$. Since $\mathbf{I}^2$ and $\mathbf{I}^4$ remain
intact in $G-\{u,w\}$, the isochromatic property of
$\mathbf{I}$ implies that the edge $e$ connecting
$\mathbf{I}^2$ to $\mathbf{H}$ has the same colour as the edge
$f$ connecting $\mathbf{I}^4$ to $\mathbf{H}$. Since
$\mathbf{H}$ is heterochromatic, we conclude that $u$ lies in
$\mathbf{H}$. As a consequence, $\mathbf{I}^1$ remains intact
in $G-\{u,w\}$ too, which means that the two edges $e'$ and
$f'$ joining $\mathbf{H}$ to $\mathbf{I}^1$ are equally
coloured as well. But then the induced colouring of
$\mathbf{H}$ yields a $3$-edge-colouring of the Petersen graph
with a single vertex removed. This contradiction proves that
$\rho(G)>2$ and $\omega(G)\ge 4$.

Class~3a contains graphs 5 and 6, Class~3b contains graphs 7
and 23, and Class~3c contains graphs 9 and 26.

\subsection*{Class 4: \textbf{H + 3I + T + 1}\, (10 graphs) }
%Group 3: 1xH, 1xT, 3xI}
\begin{figure}[htbp]%
  \begin{center}
    \subfigure[Class 4a (4 graphs)]{\scalebox{0.4}{\input{figures/sk4ane}}}
    \hspace{1cm}
    \subfigure[Class 4b (6 graphs)]{\scalebox{0.4}{\input{figures/sk4bne}}}\\[-10pt]
    \caption{Class 4}%
    \label{fig:class4}%
  \end{center}
\end{figure}
Class~4 has two subclasses 4a and 4b, both shown in
Figure~\ref{fig:class4}. Every graph $G$ from Class~4 contains
disjoint copies of $\mathbf{Z}_2$ and $\mathbf{Z}_4$, so
$\rho(G)\ge 3$ and $\omega\ge 4$.

Class~4a consists of graphs 2, 3, 13, and 14, and Class~4b
contains graphs 8, 12, 16, 20, 22, and 25.

\begin{figure}[htbp]%
  \begin{center}
    \subfigure[Class 5a (1 graph)]{\scalebox{0.4}{\input{figures/sk5ab}}}
    \hspace{1cm}
    \subfigure[Class 5b
    (1 graph)]{\scalebox{0.4}{\input{figures/sk5bb}}}\\[-10pt]
    \caption{Class 5}%
    \label{fig:class5}%
  \end{center}
\end{figure}

\subsection*{Class 5: \textbf{5I + 4}\, (2 graphs)}
Class~5 has two subclasses Class~5a and Class~5b; they are
represented in Figure~\ref{fig:class5}. Observe that every member
of Class~5 contains a $5$-pole isomorphic to $\mathbf{Z}_3$,
namely $[\mathbf{I}^1\cup\mathbf{I}^2\cup\{v_1\}]$, whose
resistance is~1 by Lemma~\ref{lemma:Z}. Hence, every member of
Class~5 is indeed a snark.

Now, let us consider a graph $G$ from Class~5a. To prove that
$\omega(G)\ge 4$ suppose the contrary. Then
$\omega(G)=\rho(G)=2$, which means that there exist vertices
$u$ and $w$ in $G$ such that $G-\{u,w\}$ is colourable. Since
each of the $5$-poles $[\mathbf{I}^k\cup\mathbf{I}^{k+1}\cup
\{v_k\}]$, with $k\in\{1,2,3,4\}$, is isomorphic to
$\mathbf{Z}_3$, and $\rho(\mathbf{Z}_3)=1$, none of them
survives in $G-\{u,w\}$. It follows that one of $u$ and $w$
lies in $\mathbf{I}^2$ and the other one lies in
$\mathbf{I}^4$. Let $\varphi$ be a $3$-edge-colouring of
$G-\{u,w\}$. Without loss of generality we may assume that
$\varphi(f_1)=1$, $\varphi(g_1)=2$ and $\varphi(h_1)=3$. Since
$\mathbf{I}^1$ remains intact in $G-\{u,w\}$, we have
$\varphi(e_1)=1$. Similarly, since $\mathbf{I}^5$ is intact, we
have $\varphi(h_2)=\varphi(h_1)=3$. Consider the colours of
$e_2$ and $f_2$. Since $\mathbf{I}^3$ remains intact
$G-\{u,w\}$, we deduce that $\varphi(e_2)=\varphi(g_2)$ and
hence $\varphi(e_2)\neq\varphi(f_2)$. Furthermore,
$\{1,2\}=\{\varphi(e_1),\varphi(g_1)\}\ne\{\varphi(e_2),\varphi(f_2)\}$
because otherwise the induced colouring of $\mathbf{I}^2$ would
give rise to a $3$-edge-colouring of the Petersen graph with
one vertex removed, which is impossible. Therefore
$\{\varphi(e_2),\varphi(f_2)\}=\{1,3\}$ or $\{2,3\}$, but since
$\varphi(f_2)\ne\varphi(h_2)=3$ we conclude that
$\varphi(e_2)=3$. However, the isochromatic property of
$\mathbf{I}^3$ forces $\varphi(g_2)=\varphi(e_2)=3$, which is
in conflict with the value $\varphi(h_2)=3$ because $g_2$ and
$h_2$ are adjacent. Therefore $\rho(G)\ge 3$ and $\omega(G)\ge
4$.

Next, let us consider a graph $G$ from Class~5b. We wish to
prove that $\omega(G)\ge 4$. Suppose the contrary. Then
$\rho(G)=2$, which again means that there exist vertices $u$
and $w$ in $G$ such that $G-\{u,w\}$ is $3$-edge-colourable. In
this case it is easy to see that one of the vertices belongs
$\mathbf{I}^2$ and the other one belongs to
$\mathbf{I}^4\cup\mathbf{I}^5\cup\{v_4\}$. Suppose that the
vertex $w$ lies in $\mathbf{I}^5\cup \{v_4\}$ so that
$\mathbf{I}^4$ remains intact. Take an arbitrary
$3$-edge-colouring $\varphi$ of $G-\{u,w\}$, and let
$\varphi(f_3)=x$. Since $\mathbf{I}^1$ and $\mathbf{I}^3$ are
also left intact in $G-\{u,w\}$, we deduce that
$\varphi(e_3)=x$, $\varphi(g_2)=x$, and finally
$\varphi(g_1)=x$. However, the edges $g_1$ and $e_3$ are
adjacent, so this is impossible. Therefore $w\in\mathbf{I}^4$,
which means that $\mathbf{I}^5$ is left intact. In this
situation we can derive a contradiction in a similar way as for
Class~5a. This proves that $\omega(G)\ge 4$ for every graph
from Class~5b.

Both Class~5a and Class~5b consist of a single graph each --
graphs 28 and 30, respectively.

\begin{figure}[htbp]
  \centerline{
   \scalebox{0.4}{
     \input{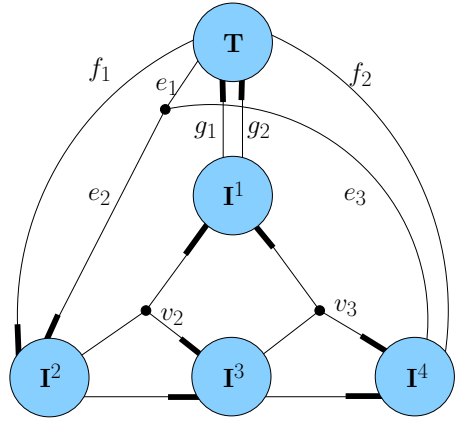}
   }
  }
   \caption{Class 6 (2 graphs)}
   \label{fig:class6}
\end{figure}

\subsection*{Class 6: \textbf{4I + T + 3}\, (2 graphs) }
The structure of Class~6 is displayed in Figure~\ref{fig:class6}.
Let us first observe that every member $G$ of Class~6 is
uncolourable because it contains the subgraph
$[\mathbf{I}^1\cup\mathbf{T}]$ isomorphic to $\mathbf{Z}_4$,
whose resistance equals 1 by Lemma~\ref{lemma:Z}. It follows
that $\rho(G)\ge 2$. To prove that $\omega(G)\ge 4$ suppose to
the contrary that $\omega(G)=2=\rho(G)$. Since each of the
subgraphs $[\mathbf{I}^k\cup\mathbf{I}^{k+1}\cup \{v_k\}]$,
with $k\in\{2,3\}$, is also uncolourable, there exist vertices
$u\in\mathbf{I}^3$ and $w\in [\mathbf{I}^1\cup\mathbf{T}]$ such
that $G-\{u,w\}$ is $3$-edge-colourable. Let $\varphi$ be one
such colouring. Without loss of generality we may assume that
$\varphi(e_1)=1$, $\varphi(e_2)=2$, and $\varphi(e_3)=3$. Since
$\mathbf{I}^2$ and $\mathbf{I}^4$ are left intact, we have
$\varphi(f_1)=2$ and $\varphi(f_2)=3$. Put together, the edges
of the $3$-connector of $\mathbf{T}$ receive three distinct
colours from the colouring $\varphi$. As previously mentioned,
every $3$-edge-colouring of $\mathbf{T}$ forces a repeated
colour in the $3$-connector. Therefore $w$ must belong to
$\mathbf{T}$. It follows that $\mathbf{I}^1$ remains intact and
therefore $\varphi(g_1)=\varphi(g_2)$. But then the induced
colouring of $\mathbf{T}$ yields a $3$-edge-colouring of the
Petersen graph with a single vertex removed. This contradiction
proves that $\omega(G)\ge 4$.

Class~6 contains two nonisomorphic graphs -- 29 and 31.

\subsection*{Properties of graphs in $\mathcal{M}$}

We have determined the values of several invariants for the 31
snarks in $\mathcal{M}$. In most cases the computations were
performed by a computer. The evaluated invariants can be
divided into two groups. The first group is constituted by
general invariants: namely, the order of the automorphism
group, \textit{genus} (minimum genus of an orientable surface
upon which a given graph can be drawn without intersections), diameter, radius, and \textit{circumference} (the
maximal circuit length in a graph). The values of these
invariants for individual members of $\mathcal{M}$ are
summarised in Table~\ref{table:31_snarks-general invariants}.
In particular, all members of $\mathcal{M}$ have circumference
$41$, that is, $n-3$ where $n$ is the number of vertices. The
values of the remaining invariants vary over the set
$\mathcal{M}$. It is quite remarkable that the automorphism
group of every graph in $\mathcal{M}$ is a $2$-group (or is
trivial).

%%%%%%%%%%%%%%%%%%%%%%%%%%%%%%%%%%%%%%%%%%%%%%%%%%%%%%%

	\begin{table}[htb!]
	\small
		\centering % used for centering table
		 \renewcommand{\arraystretch}{1.1} %1 standard
%		\begin{tabular}{cccc} % centered columns (4 columns)
		\begin{tabular}{crrrrrr} % centered columns (4 columns)
			\toprule %inserts double horizontal lines
Class & Graph number & $|\Aut(G)|$  & Genus & Diameter & Radius &
Circumference\\
%[0.5ex]%inserts table
			%heading
			\midrule % inserts single horizontal line
1a & 15 & 4 & 4 & 8 & 6 & 41 \\
1a & 17 & 64 & 5 & 8 & 7 & 41\\
1a & 18 & 8 & 5 & 8 & 7 & 41\\
\hline
1b & 1 & 16 & 5 & 8 & 6 & 41\\
1b & 4 & 1 & 4 & 8 & 6 & 41\\
1b & 21 & 4 & 5 & 8 & 6 & 41\\
1b & 24 & 4 & 5 & 8 & 6 & 41\\
\hline
2 & 10 & 2 & 4 & 8 & 5 & 41\\
2 & 11 & 2 & 4 & 8 & 5 & 41\\
2 & 19 & 16 & 5 & 8 & 5 & 41\\
2 & 27 & 2 & 5 & 8 & 5 & 41\\
\hline
3a & 5 & 4 & 5 & 8 & 6 & 41\\
3a & 6 & 1 & 4 & 8 & 6 & 41\\
3b & 7 & 2 & 5 & 8 & 6 & 41\\
3b & 23 & 8 & 5 & 8 & 6 & 41\\
\hline
3c & 9 & 4 & 4 & 8 & 6 & 41\\
3c & 26 & 1 & 4 & 8 & 6 & 41\\
\hline
4a & 2 & 8 & 5 & 8 & 7 & 41\\
4a & 3 & 4 & 5 & 8 & 7 & 41\\
4a & 13 & 2 & 4 & 8 & 6 & 41\\
4a & 14 & 1 & 4 & 8 & 6 & 41\\
\hline
4b & 8 & 1 & 4 & 8 & 6 & 41\\
4b & 12 & 1 & 4 & 8 & 6 & 41\\
4b & 16 & 4 & 5 & 8 & 6 & 41\\
4b & 20 & 1 & 5 & 8 & 6 & 41\\
4b & 22 & 4 & 5 & 8 & 6 & 41\\
4b & 25 & 4 & 5 & 8 & 6 & 41\\
\hline
5a & 28 & 2 & 4 & 8 & 6 & 41\\
\hline
5b & 30 & 2 & 5 & 7 & 6 & 41\\
\hline
6 & 29 & 2 & 5 & 7 & 6 & 41\\
6 & 31 & 1 & 4 & 8 & 6 & 41\\
			%[1ex] % [1ex] adds vertical space
			\bottomrule %inserts single line
		\end{tabular}
		\caption{Basic invariant values for the 31 snarks of
$\mathcal{M}$.
}
        \label{table:31_snarks-general invariants}
	\end{table}	

%%%%%%%%%%%%%%%%%%%%%%%%%%%%%%%%%%%%%%%%%%%%%%%%%%%%%%%

The second group of invariants comprises those which are
of particular interest for snarks: perfect matching index
$\pi$, resistance $\rho$, weak oddness $\omega'$, and two
invariants introduced in~\cite{LiSt, Steffen:pmcovers,
Steffen:5-flows} and denoted by $\mu_3$ and~$\gamma_2$; see
also \cite{FMS-survey} for a recent survey. The
\textit{perfect matching index} of a bridgeless cubic graph
$G$, denoted by $\pi(G)$ (also known as \textit{excessive
index} and denoted by $\chi'_e$), is the smallest number of
perfect matchings that cover all the edges of $G$~\cite{BoCa,
FouqVa}. The \textit{weak oddness} of a cubic graph $G$,
denoted by $\omega'(G)$, is the smallest number of components
of odd order in an even factor of $G$; by an \textit{even
factor} we mean a spanning subgraph with all degrees even.
Given a bridgeless cubic graph $G$, we define $\gamma_2(G)$ to
be the smallest number of common edges that two perfect
matchings of $G$ can have, and let $\mu_3(G)$ be the smallest
number of edges of $G$ that are left uncovered by the union of
any three perfect matchings of $G$.

The perfect matching index $\pi(G)$ of $G$ is bounded below by
$3$ and equals $3$ if and only if $G$ is $3$-edge-colourable.
It is believed, by a conjecture of Berge (see~\cite{S}), that
$\pi(G)\le 5$ for every bridgeless cubic graph $G$. This
conjecture, if true, thus divides all snarks into two
subclasses, those with perfect matching index equal to $4$, and
those with perfect matching index $5$. We have determined that
$\pi(G)=4$ for every $G\in\mathcal{M}$.

The remaining invariants, along with oddness, can be regarded
as measures of uncolourability as they take value $0$ on
$3$-edge-colourable graphs, and positive values otherwise.
Their comparison with oddness and resistance may therefore be
very instructive.

First of all, using a computer we have determined that
$\rho(G)=3$ for all snarks $G\in\mathcal{M}$. As regards weak
oddness, it is easy to see that $\omega'(G)$ is an even integer
such that $\rho(G)\le\omega'(G)\le\omega(G)$.
Furthermore, if $\omega(G)=4$, then necessarily $\omega'(G)=4$
as well, because otherwise $\omega'(G)=2$ would immediately
yield $\rho(G)=2$ whence $\omega(G)=2$. In particular, for all
snarks $G\in\mathcal{M}$ we have $\omega'(G)=\omega(G)=4$.
It is known, however, that in general both the difference
$\omega(G)-\omega'(G)$ and $\omega'(G)-\rho(G)$ can be
arbitrarily large (see~\cite{Al, LM}).

We finish this section by discussing the previously mentioned
invariants $\gamma_2$ and $\mu_3$. In
\cite[Proposition~2.1]{Steffen:5-flows} Steffen proved  that
$\omega(G)\le 2\gamma_2(G)$ for every bridgeless cubic
graph~$G$. Using a computer we have determined that
$\gamma_2(G)=2$ for every $G\in\mathcal{M}$, which shows that
every snark from $\mathcal{M}$ fulfils the upper bound on
$\omega$ set by $\gamma_2$ with equality. Similarly, in
\cite[Corollary~2.4]{LiSt} Jin and Steffen proved that
$\omega(G)\le 2\mu_3(G)/3$ for every bridgeless cubic graph.
Using a computer we have determined that all snarks
$G\in\mathcal{M}$ have $\mu_3(G)=6$, which means that they
again reach the upper bound on $\omega$ in terms of $\mu_3$
with equality. Snarks with the latter property have a very
special structure of sets of edges left uncovered by three
perfect matchings and therefore deserve special attention.

The values of perfect matching index, resistance, weak oddness,
$\gamma_2$, and $\mu_3$ for the snarks of $\mathcal{M}$ are
summarised in Table~\ref{table:31_snarks-special invariants}.

\begin{table}[htb!]
		\centering % used for centering table
		 \renewcommand{\arraystretch}{1.1} %1 standard
%		\begin{tabular}{cccc} % centered columns (4 columns)
		\begin{tabular}{crrrr} % centered columns (4 columns)
			\toprule %inserts double horizontal lines		
$\pi$ & $\rho$   & $\omega'$ & $\gamma_2$ & $\mu_3$ \\
%[0.5ex] % inserts table
			%heading
			\midrule % inserts single horizontal line
4 & 3 & 4 & 2 & 6 \\
			\bottomrule %inserts single line
		\end{tabular}
	\caption{Values of perfect matching index and
unclourability measures for the snarks of
$\mathcal{M}$}\label{table:31_snarks-special invariants}
	\end{table}	

\subsection*{Infinite families}
Each of the six classes described above gives rise to an
infinite family of snarks with oddness at least $4$ and cyclic
connectivity $4$. It is sufficient to replace the basic
building blocks $\mathbf{I}$, $\mathbf{H}$, $\mathbf{T}$, and
$\mathbf{N}$ obtained from the Petersen graph with similar
structures created from any cyclically $4$-edge-connected
snark. With a little additional care one can construct infinite
families of snarks with increasing oddness.

\section{Completeness of $\mathcal{M}$}\label{sec:complete}

In this section we prove that the set $\mathcal{M}$,
constructed and analysed in Section~\ref{sec:31}, is the
complete set of pairwise nonisomorphic snarks with cyclic
connectivity $4$, oddness at least $4$, and minimum order. Our
point of departure is the following theorem proved in~\cite{GMS1}.

\begin{theorem}\label{thm:44}
The smallest number of vertices of a snark with cyclic
connectivity $4$ and oddness at least $4$ is $44$. The girth of
each such snark is at least $5$.
\end{theorem}

This result is a consequence of the following stronger and more
detailed result from~\cite{GMS1} which will be needed for the
proof of the main result of this paper.

\begin{theorem} \label{thm:4decomp-detailed}
Let $G$ be a snark with oddness at least $4$, cyclic
connectivity $4$, and minimum number of vertices. Let $S$ be a
cycle-separating $4$-edge-cut in $G$ whose removal leaves
components $G_1$ and $G_2$. Then, up to permutation of the
index set $\{1,2\}$, exactly one of the following occurs.
\begin{itemize}
\item[{\rm (i)}] Both $G_1$ and $G_2$ are uncolourable, in
    which case each of them can be extended to a cyclically
    $4$-edge-connected snark by adding two vertices.
\item[{\rm (ii)}]  $G_1$ is uncolourable and $G_2$ is
    heterochromatic, in which case $G_1$ can be extended to
    a cyclically $4$-edge-connected snark by adding two
    vertices, and $G_2$ can be extended to a cyclically
    $4$-edge-connected snark by adding two isolated edges.
\item[{\rm (iii)}]$G_1$ is uncolourable and $G_2$ is
    isochromatic, in which case $G_1$ can be extended to a
    cyclically $4$-edge-connected snark by adding two
    vertices, and $G_2$ can be extended to a cyclically
    $4$-edge-connected snark by adding two vertices, except
    possibly $\zeta(G_2)=2$. In the latter case, $G_2$ is a
    partial junction of two colour-open $4$-poles, which
    may be isochromatic or heterochromatic in any
    combination.
\end{itemize}
\end{theorem}

Here is our main result:

\begin{theorem}\label{thm:main}
The set $\mathcal{M}$ is the complete set of snarks with cyclic
connectivity $4$ and oddness at least $4$ on $44$ vertices.
\end{theorem}

Let $\mathcal{N}$ be the set of all snarks with cyclic
connectivity $4$ and oddness at least $4$ on $44$ vertices. To
prove Theorem~\ref{thm:main} it suffices to show that
$\mathcal{N}\subseteq\mathcal{M}$. The general strategy of the
proof is to show that every snark $G\in\mathcal{N}$ is a
$4$-join of two cyclically $4$-edge-connected snarks of order
at most 36. As soon as this is done, one can perform $4$-joins
in all possible ways that give rise to a cyclically
$4$-edge-connected snark of order $44$, identify those whose
oddness equals $4$, and check whether all of them belong to
$\mathcal{M}$.

In order to apply this strategy we employ
Theorem~\ref{thm:4decomp-detailed}. It implies that we can
split each $G\in\mathcal{N}$ into two subgraphs $G_1$ and $G_2$
each of which can be extended to cyclically a
$4$-edge-connected snark by adding at most two vertices. The
difficult part of the proof arises when one of the subgraphs,
namely $G_1$, has 36 vertices and $G_2$ is an isochromatic
$4$-pole on $8$ vertices (see (iii) of
Theorem~\ref{thm:4decomp-detailed}). Adding two adjacent
vertices to $G_1$ is now useless because the list of all
cyclically $4$-edge-connected snarks is known only up to 36
vertices~\cite{BrGHM,GMS1}. Instead, we show that it is
possible to add two isolated edges to $G_1$ in such a way that
a cyclically $4$-edge-connected snark of order 36 is created. A
detailed analysis that precedes this step is the core of the
proof of Theorem~\ref{thm:main}, which now follows.

\begin{proof}
Let $G$ be a snark with cyclic connectivity 4 and oddness at
least 4 on 44 vertices. We wish to prove that
$G\in\mathcal{M}$. Suppose the contrary. In order to derive a
contradiction we first establish four claims.

\medskip\noindent
\textbf{Claim 1.} \textit{Every cycle-separating $4$-edge-cut
of $G$ determines two components, one uncolourable on $36$ and
one isochromatic on $8$ vertices. Both components are
$2$-edge-connected.}

\medskip\noindent
Proof of Claim~1. Let $S$ be an arbitrary cycle-separating
$4$-edge-cut in $G$ and let $G_1$ and $G_2$ be the components
of $G-S$. Since $S$ is a minimum cycle-separating edge-cut,
both $G_1$ and $G_2$ are easily seen to be $2$-edge-connected.
>From Theorem~\ref{thm:4decomp-detailed} we deduce that one of
the components, say $G_1$, is uncolourable. If $G_2$ was either
uncolourable or heterochromatic, then it would have at least 10
vertices and hence $G_1$ has at most $34$ vertices. Using
Theorem~\ref{thm:4decomp-detailed} again we could conclude that
both $G_1$ and $G_2$ can be extended to snarks of order at most
$36$, so $G$ would be a $4$-join of two snarks from
$\mathcal{S}_{36}$ and therefore a member of $\mathcal{M}$.
This contradiction shows that $G_2$ is isochromatic.

Now we prove that the isochromatic component produced by an
arbitrary cycle-sepa\-rating $4$-edge-cut has only eight
vertices. Suppose to the contrary that there exists a cycle
separating edge-cut $R$ in $G$ such that the isochromatic
component, denoted by $J$, has at least ten vertices. Choose
$R$ to minimise the number of vertices of $J$. Clearly, the
other component $K$ of $G-R$ has at most $34$ vertices and, by
the first part of the proof, it is uncolourable. It follows
that $K$ also has at least ten vertices, so $J$ has at most
$34$ vertices too. Theorem~\ref{thm:4decomp-detailed} further
implies that $K$ can be extended to a cyclically
$4$-edge-connected snark $\tilde K$ by adding two vertices. If
$\zeta(J)\ge 3$, the same theorem implies that $J$ can also be
extended to a cyclically $4$-edge-connected snark $\tilde J$ by
adding two vertices. Since both $\tilde J$ and $\tilde K$ have
order at most 36, $G$ belongs to $\mathcal{M}$ -- a
contradiction. Therefore $J$ has a $2$-edge-cut. By
Theorem~\ref{thm:4decomp-detailed}~(iii), $J$ is a partial
junction of two colour-open $4$-poles $J_1$ and~$J_2$. Since
both $\delta_G(J_1)$ and $\delta_G(J_2)$ are cycle-separating
$4$-edge-cuts, both $J_1$ and $J_2$ must be isochromatic. If
any of them had more than eight vertices, then the
corresponding edge-cut would contradict the choice of $R$.
Therefore $J$ is a partial junction of two copies of the
isochromatic $4$-pole $\mathbf{I}$ on eight vertices, so $J$
has 16 vertices and $K$ has 28 vertices. It follows that
$G$ can be expressed in the form $(P\odot P)\otimes H$ where
$P$ is the Petersen graph, $H$ is a snark on 30 vertices, and
$L_1\odot L_2$ denotes a $4$-join of cubic graphs $L_1$ and
$L_2$ which employs $4$-poles resulting from the removal of two
adjacent vertices from both $L_1$ and~$L_2$, while $Q_1\otimes
Q_2$ denotes a $4$-join of cubic graphs $Q_1$ and $Q_2$ which
employs a $4$-pole obtained from $Q_1$ by removing two
nonadjacent edges and a $4$-pole obtained from $Q_2$ by
removing two adjacent vertices. Using a computer we have
constructed all graphs $G$ arising in this way and verified
that each of them either belongs to $\mathcal{M}$ or has
oddness at most~$2$. This contradiction establishes Claim~1.

\medskip\noindent
\textbf{Claim 2.} \textit{Every $4$-edge-cut in $G$ separates a
subgraph with at most eight vertices from the rest of $G$.}

\medskip\noindent
Proof of Claim~2. Let $Q$ be a $4$-edge-cut in $G$. If one of
the components of $G-Q$ is acyclic, then, by
Lemma~\ref{lemma:acyclic}, this component has two vertices. If
$Q$ is cycle-separating, then the conclusion follows from
Claim~1. Claim~2 is proved.

\medskip

Fix a cycle-separating $4$-edge-cut $S$ in $G$; we will refer
to $S$ as the \textit{principal} $4$-edge-cut of~$G$. As
before, let $G_1$ and $G_2$ be the components of $G-S$ where
$G_1$ is uncolourable. By Claim~1, the other component $G_2$ is
isochromatic on eight vertices. Let $A$ be the set of
end-vertices of $S$ in $G_1$. Since $S$ is independent, we have
$|A|=4$.

\medskip

The remainder of the proof is devoted to proving the following
fact:
\begin{itemize}
\item[] \textit{The component $G_1$ can be extended to a cyclically
$4$-edge-connected snark  by adding two edges between the
vertices of $A$.}
\end{itemize}
As a consequence, $G$ will be a $4$-join of
two graphs from $\mathcal{S}_{36}$, and therefore a member of
$\mathcal{M}$. This will provide a final contradiction.

\medskip\noindent
\textbf{Claim 3.}\textit{$G_1$ is cyclically $3$-edge-connected
and has a cycle-separating $3$-edge-cut.}

\medskip\noindent
Proof of Claim~3. If $\zeta(G_1)\ge 4$, then adding to $G_1$
two edges joining the vertices of $A$ in an arbitrary manner
would produce a cyclically $4$-edge-connected snark of order
36. Consequently, $G$ would belong to $\mathcal{M}$. Therefore
$\zeta(G_1)\le 3$. To prove that $\zeta(G_1)\ge 3$ suppose to
the contrary that $G_1$ has a cycle-separating 2-edge-cut $U$,
and let $G_{11}$ and $G_{12}$ be the two components of $G_1-U$.
Since $G$ is cyclically $4$-edge-connected, two edges of $S$
join $G_2$ to $G_{11}$ and other two edges of $S$ join $G_2$ to
$G_{12}$. Thus $\delta_G(G_{11})$ is a cycle-separating
$4$-edge-cut which separates $G_{11}$ from $G_{12}\cup G_2$.
Since $G_{12}\cup G_2$ contains more then eight vertices,
Claim~1 implies that $G_{11}$ is an isochromatic dipole on
eight vertices. Likewise, $G_{12}$ is an isochromatic dipole on
eight vertices. Put together, $G$ has altogether $24$ vertices,
which is again a contradiction. Therefore $\zeta(G_1)=3$.
Finally, $G_1$ is homeomorphic to a certain cubic graph $\bar
G_1$ with $\zeta(\bar G_1)=\zeta(G_1)$ on 32 vertices, so $G_1$
is not a subdivision of the complete graph $K_4$ and therefore
contains a cycle-separating $3$-edge-cut. This establishes
Claim~3.

\medskip

Before we can proceed we need two definitions. First, a
cycle-separating $3$-edge-cut $Q$ in $G_1$ will be called
\textit{balanced} if each component of $G_1-Q$ is incident with
exactly two edges of the principal $4$-edge-cut $S$. Otherwise,
$Q$ will be called \textit{unbalanced}. Second, let $R$ and $T$
be two cycle-separating $3$-edge-cuts in $G_1$. Let $X_1$ and
$X_2$ be the components of $G_1-R$ and let $Y_1$ and $Y_2$ be
the components of $G_1-T$. The edge-cuts $R$ and $T$ in $G_1$
will be called \textit{comparable} if $X_i\subseteq Y_j$ or
$Y_j\subseteq X_i$ for some $i,j\in\{1,2\}$.

\medskip\noindent
\textbf{Claim 4.} \textit{$G_1$ contains two incomparable balanced
$3$-edge-cuts.}

\medskip\noindent
Proof of Claim~4. Let us first observe that if $Q$ is an
arbitrary unbalanced $3$-edge-cut in $G_1$, then adding any two
edges between the vertices of $A$ in an arbitrary manner will
produce a cubic graph $\tilde G_1$ where $Q$ has ceased to be
an edge-cut. It follows that if every cycle-separating
$3$-edge-cut in $G_1$ is unbalanced, then $\tilde G_1$ is a
cyclically $4$-edge-connected snark and $G$ is a $4$-join of
$\tilde G_1$ with the Petersen graph. Since $G_1$ has $36$
vertices, $G\in\mathcal{M}$ and we have arrived at a
contradiction. Therefore $G_1$ must contain a balanced
$3$-edge-cut.

\begin{figure}[htbp]
  \centerline{
   \scalebox{0.4}{
     \input{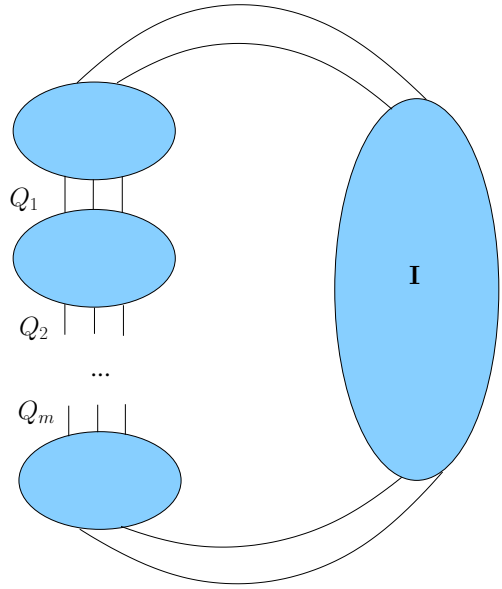}
   }
  }
   \caption{Comparable balanced $3$-edge-cuts}
   \label{fig:comparable}
\end{figure}

If every pair of balanced $3$-edge-cuts in $G_1$ is comparable,
we can arrange the balanced $3$-edge-cuts in an increasing
linear order, say $Q_1, \ldots, Q_m$. Clearly, there is a
component $H_1$ of $G_1-Q_1$ and a component $H_m$ of $G_1-Q_m$
such that both of them are disjoint from all of $Q_1, \ldots,
Q_m$. It is easy to see that $H_1$ contains two vertices of $A$
while $H_m$ contains the other two, see
Figure~\ref{fig:comparable}. Thus we can connect the two vertices
of $H_1\cap A$ to those of $H_m\cap A$ by two edges, producing
a cyclically $4$-edge-connected snark $\tilde G_1$. Again, $G$
is a $4$-join of $\tilde G_1$ with the Petersen graph, so $G$
belongs to $\mathcal{M}$ contrary to the assumption. Therefore
$G_1$ contains two incomparable balanced $3$-edge-cuts. This
proves Claim~4.

\begin{figure}[htbp]
  \centerline{
   \scalebox{0.6}{
     \input{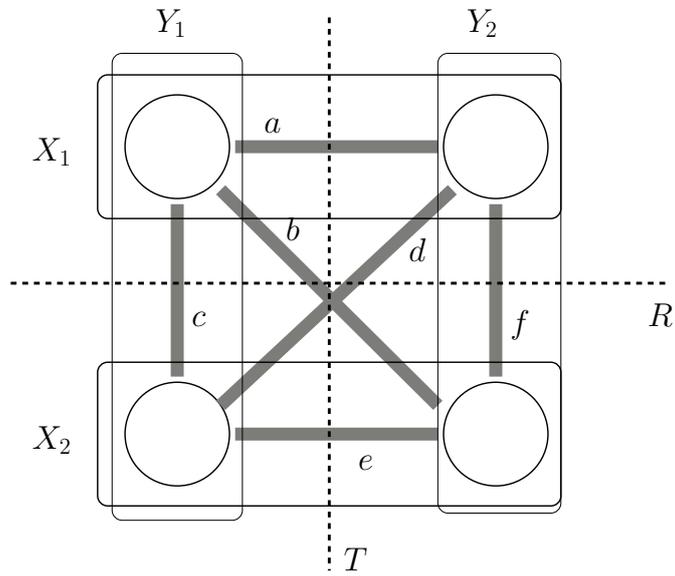}
   }
  }
   \caption{Crossing edge-cuts $R$ and $T$}
   \label{fig:sets}
\end{figure}

\bigskip

In the rest the proof we explore the structure of $G_1$ arising
from a pair of incomparable balanced $3$-edge-cuts. Let $R$ and
$T$ be any two incomparable balanced 3-edge-cuts in $G_1$.
Clearly, $G_1-R$ has two components, say $X_1$ and $X_2$, and
$G_1-T$ has two components, say $Y_1$ and $Y_2$. The definition
of comparable edge-cuts readily implies that each of the
subgraphs $X_i\cap Y_j$ is non-empty. Let $a$ be the number of
edges between $X_1\cap Y_1$ and $X_1\cap Y_2$, $b$~the number
of edges between $X_1\cap Y_1$ and $X_2\cap Y_2$, $c$~the
number of edges between $X_1\cap Y_1$ and $X_2\cap Y_1$, $d$
the number of edges between $X_1\cap Y_2$ and $X_2\cap Y_1$,
$e$ the number of edges between $X_2\cap Y_1$ to $X_2\cap Y_2$,
and finally $f$ the number of edges between $X_1\cap Y_2$ and
$X_2\cap Y_2$; see Figure~\ref{fig:sets}.

Since $X_1$, $X_2$, $Y_1$ and $Y_2$ are all connected, we have
\begin{align}\label{eq:aecf}
a\ge 1, \quad e\ge 1, \quad c\ge 1, \quad f\ge 1.
\end{align}
Next,
\begin{align}\label{eq:cutR}
|R|=|\delta_{G_1}(X_1)|=c+f+b+d=3
\end{align}
and
\begin{align}\label{eq:cutT}
|T|=|\delta_{G_1}(Y_1)|=a+e+b+d=3.
\end{align}
By combining (\ref{eq:aecf}) and (\ref{eq:cutR}) we can further
conclude that
\begin{align}\label{eq:bd}
b+d\le 1.
\end{align}

We now consider two cases according to whether there exists a
set $X_i\cap Y_j$ such that $X_i\cap X_j\cap A=\emptyset$ or
not.

\medskip

\noindent \textbf{Case 1:} \textit{There exists a subgraph
$X_i\cap Y_j$, with $i,j\in\{1,2\}$, such that $X_i\cap X_j\cap
A=\emptyset$.} In view of symmetry we can clearly assume that
$X_1\cap X_2\cap A=\emptyset$. In this situation $X_1\cap Y_2$
is incident with two edges of $S$, because $R$ is balanced, and
$X_2\cap Y_1$ is incident with the other two edges of $S$,
because $T$ is balanced. It follows that $|X_2\cap Y_2\cap
A|=0$ as well.

With (\ref{eq:aecf}) in mind, we first prove that $a=e=c=f=1$.
Suppose to the contrary that one of these values is strictly
greater than $1$. In view of symmetry we can assume that $a>1$.
Now (\ref{eq:aecf}) and (\ref{eq:cutT}) forces $a=2$, $e=1$,
$b=0$, and $d=0$, which together with (\ref{eq:cutR}) implies
that either $c=2$ and $f=1$, or $c=1$ and $f=2$. In the former
case, $|\delta_G(X_2\cap Y_2)|=b+e+f=2$ contrary to the fact
that $G$ is $3$-edge-connected, while in the latter case
$X_2\cap Y_1$ is separated from the rest of $G_1$ by $c+d+e=2$
edges, contradicting Claim~3. Therefore $a=e=c=f=1$.

Now, (\ref{eq:cutR}) and (\ref{eq:bd}) imply that $b+d=1$.
However, $b\ge 1$ for otherwise $|\delta_G(X_1\cap
Y_1)|=a+b+c=2$, contrary to the fact that $G$ is
$3$-edge-connected. So $b=1$ and $d=0$. It follows that
$|\delta_G(X_1\cap Y_1)|=a+b+c=3$ and $|\delta_G(X_2\cap
Y_2)|=b+e+f=3$. Since $G$ is cyclically $4$-edge-connected,
both $X_1\cap Y_1$ and $X_2\cap Y_2$ are acyclic, and
therefore, by Lemma~\ref{lemma:acyclic}, both consist of a
single vertex. Further, $|\delta_G(X_1\cap Y_2)|=2+a+d+f=4$ and
$|\delta_G(X_2\cap Y_1)|=2+c+d+e=4$, so $|X_1\cap Y_2|\ge 2$
and $|X_2\cap Y_1|\ge 2$. By Claim~2, one of the components
determined by $\delta_G(X_1\cap Y_2)$ has at most eight
vertices. Since the number of vertices of $G$ lying outside
$X_1\cap Y_2$ is
$$|G_2|+|X_2\cap Y_1|+|X_1\cap Y_1|+|X_2\cap Y_2|\ge
8+2+1+1=12,$$ we conclude that $|X_1\cap Y_2|\le 8$. Similarly,
$|X_2\cap Y_1|\le 8$. Summing up,
$$|G_1|\le |X_1\cap Y_1|+|X_2\cap Y_2| + |X_1\cap Y_2|+|X_2\cap Y_1|\le
1+1+8+8<36,$$ which contradicts Claim~1. This establishes Case~1.

\medskip

\noindent \textbf{Case 2:} \textit{$X_i\cap X_j\cap
A\ne\emptyset$ for each $i,j\in\{1,2\}$.} Clearly, this is only
possible when each of the subgraphs $X_i\cap Y_i$ is incident
with exactly one edge of the principal $4$-edge-cut $S$. We may
assume that $A=\{a_1,a_2,a_3,a_4\}$ where $a_1$ lies in
$X_1\cap Y_1$, $a_2$ lies in $X_1\cap Y_2$, $a_3$ lies in
$X_2\cap Y_1$, and $a_4$ lies in $X_2\cap Y_2$.

For each $X_i\cap Y_i$ we count the number of edges of $R\cup
T$ this subgraph is incident with.  According to
(\ref{eq:aecf}) we obtain
\begin{align}\label{eq:cobdr}
a+b+c\ge 2, \quad a+d+f\ge 2, \quad c+d+e\ge 2, \quad b+e+f\ge 2.
\end{align}
If all the inequalities in (\ref{eq:cobdr}) are strict, then summing
the first two of them yields
$$6\le 2a+(b+c+d+f)=2a+|R| = 2a +3,$$
whence $a\ge 2$. Similarly, summing the latter two inequalities
implies that $e\ge 2$. But then $$|T|=a+b+d+e\ge 4,$$ which is
a contradiction.

Therefore at least one of the inequalities in (\ref{eq:cobdr})
holds with equality. Due to symmetry we may assume that
$a+b+c=2$. From (\ref{eq:aecf}) we now infer that $a=c=1$ and
$b=0$. If we plug these values into (\ref{eq:cutR}) and
(\ref{eq:cutT}), we get $d+f=2=d+e$. Since $d\le 1$ on account
of (\ref{eq:bd}), we conclude that either $d=e=f=1$, or $d=0$
and $e=f=2$.

First assume that $d=e=f=1$.  In this case $|\delta_G(X_1\cap
Y_1)| = 3 = |\delta_G(X_2\cap Y_2)|$, so $|X_1\cap Y_1| = 1
=|X_2\cap Y_2|$ by Lemma~\ref{lemma:acyclic}. Furthermore,
$|\delta_G(X_1\cap Y_2)| = |\delta_G(X_2\cap Y_1)| =4$. By
Claim~2, one of the components determined by $\delta_G(X_1\cap
Y_2)$ has at most eight vertices. As in Case~1, the number of
vertices of $G$ outside $X_1\cap Y_2$ is at least $12$, so
$|X_1\cap Y_2|\le 8$. Similarly $|X_2\cap Y_1|\le 8$, and
therefore $|G_1|\le 1+1+8+8<36$, which contradicts Claim~1.

Next assume that $d=0$ and $e=f=2$. Recall that $a=c=1$ and
$b=0$. It follows that $|\delta_G(X_1\cap Y_1)|=3$ which means that
$X_1\cap X_1$ consists of a single vertex $a_1$. Further,
$|\delta_G(X_1\cap Y_2)|=4$ and  $|\delta_G(X_2\cap Y_1)|=4$,
so $|X_1\cap Y_2|\le 8$ and $|X_2\cap Y_1|\le 8$. Also,
$|\delta_G(X_2\cap Y_2)|=5$. In fact, Claim~1 implies that
$X_1\cap Y_2$ and $X_2\cap Y_1$ are isochromatic $4$-poles on
eight vertices, because $R$ and $T$ are cycle-separating
edge-cuts and hence $X_1\cap Y_2$ and $X_2\cap Y_1$ must
contain circuits.

We show that $X_2\cap Y_2$ is $2$-edge-connected. If $X_2\cap
Y_2$ was disconnected, it would have a component $B$ with
$|\delta_G(B)|\le 2$, contradicting the fact that $G$ is
$3$-edge-connected. If there was a bridge in $X_2\cap Y_2$,
then the bridge would join a subgraph $B_1$ with
$\delta_G(B_1)=3$ to a subgraph $B_2$ with $\delta_G(B_2)=4$.
By Lemma~\ref{lemma:acyclic} and Claim~2, $|B_1|=1$ while
$|B_2|\le 8$, so $|G_1|\le 1+8+8+1+8<36$, contrary to Claim~1.
Therefore $X_2\cap Y_2$ is $2$-edge-connected.

Let us now extend $G_1$ to a cubic graph $\tilde G_1$ by adding
the edges $a_1a_4$ and $a_2a_3$. Since $G_1$ is uncolourable,
$\tilde G_1$ is a snark. We wish to show that $\tilde G_1$ is
cyclically $4$-edge-connected. To this end, observe that a
cycle-separating $3$-edge-cut of $G_1$ that separates either
$a_1$ from $a_4$ or $a_2$ from $a_3$ fails to be an edge-cut in
$\tilde G_1$. Therefore every cycle-separating $3$-edge-cut
that might survive from $G_1$ to $\tilde G_1$ separates
$\{a_1,a_4\}$ from $\{a_2,a_3\}$. We prove that such a cut does
not exist, displaying four edge-disjoint
$\{a_1,a_4\}$-$\{a_2,a_3\}$-paths in $G_1$.

\begin{figure}[htbp]
  \centerline{
   \scalebox{0.5}{
     \input{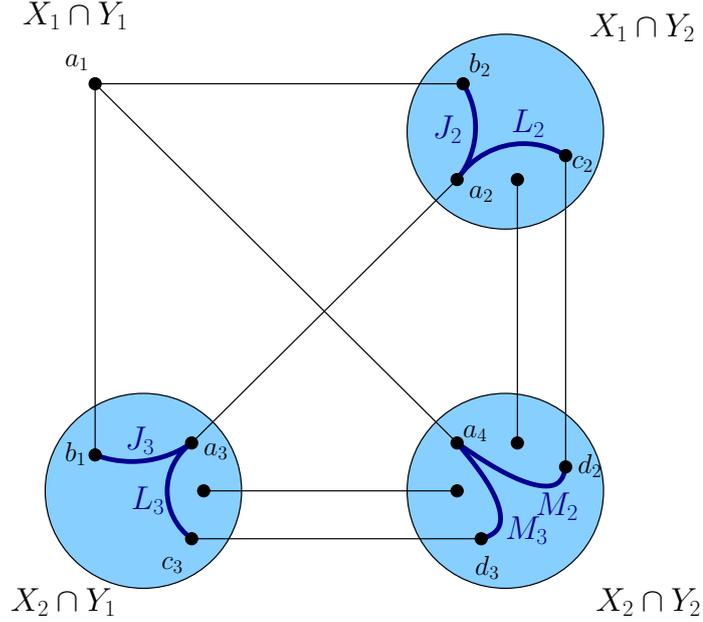}
   }
  }
   \caption{The required paths in Case 2}
   \label{fig:case2}
\end{figure}

Let $a_1b_2$ and $a_1b_3$ be the edges that join $a_1$ to a
vertex $b_2$ in $X_1\cap Y_2$ and to a vertex $b_3$ in $X_2\cap
Y_1$, respectively. Let $c_2d_2$ be an edge joining a vertex
$c_2$ in $X_1\cap Y_2$ to a vertex $d_2$ in $X_2\cap Y_2$, and
let $c_3d_3$ be an edge joining a vertex $c_3$ in $X_2\cap Y_1$
to a vertex $d_3$ in $X_2\cap Y_2$ (see Figure~\ref{fig:case2}).
Since $X_1\cap Y_2$, $X_2\cap Y_1$, and $X_2\cap Y_2$ are all
$2$-edge-connected, there exist
\begin{itemize}
\item[--] edge-disjoint paths $J_2$ and $L_2$ in $X_1\cap
    Y_2$ from $a_2$ to $b_2$ and $c_2$, respectively;

\item[--] edge-disjoint paths $J_3$ and $L_3$ in $X_2\cap
    Y_1$ from $a_3$ to $b_3$ and $c_3$, respectively; and
\item[--] edge-disjoint paths $M_2$ and $M_3$ from $a_4$ to
    $d_2$ and $d_3$, respectively.
\end{itemize}
It follows that $a_1b_2J_2\inv$, $a_1b_3J_3\inv$,
$L_2c_2d_2M_2\inv$, and $L_3c_3d_3M_3\inv$ are four
edge-disjoint paths joining $\{a_1,a_4\}$ to $\{a_2,a_3\}$ in
$G_1$. As a consequence, $\tilde G_1$ is a cyclically
$4$-edge-connected snark of order $36$, and therefore $G\in
\mathcal{M}$. This final contradiction establishes the theorem.
\end{proof}

\section{Further computational results}
\label{sect:properties}

In addition to determining the complete set $\mathcal{M}$ of
snarks of oddness at least 4 with cyclic connectivity 4 and
minimum number of vertices, we have also generated a
considerable number of snarks of oddness at least $4$ with
orders ranging from 46 to 52. In order to produce as many
nonisomorphic snarks as possible we have applied the $4$-join
operation in all possible ways to any pair of cyclically
$4$-edge-connected snarks with at most 36 vertices as long as
the resulting graph had at most 52 vertices and computed its
oddness. To extend the set we have further applied I-extensions
and I-reductions in all possible ways to the snarks of oddness
at least 4 constructed by $4$-joins until no new snarks of
oddness at least 4 were found. An I-\textit{extension} of a
cubic graph subdivides two edges $e$ and $f$ of $G$ with a new
vertex $v_e$ and $v_f$, respectively, and adds a new edge
between $v_e$ and $v_f$; an I-\textit{reduction} is the inverse
of an I-extension. The combination of $4$-joins, I-extensions,
and I-reductions produced a set 887\,152 nonisomorphic snarks of
orders from 46 to 52, all of oddness $4$. None of the produced
snarks had oddness greater than~4 or cyclic connectivity
greater than $4$.

The counts of the number of snarks of oddness at least 4 which
our method yielded can be found in
Table~\ref{table:counts_snarks_oddness4}. While the set of 31
snarks of order 44 is complete by Theorem~\ref{thm:44}, it is
very unlikely that this is the case for any of the sets of
snarks of oddness $4$ of orders between 46 and $52$.
%The graphs from Table~\ref{table:counts_snarks_oddness4} can be
%downloaded from the \textit{House of Graphs}~\cite{BCGM} at
%\url{http://hog.grinvin.org/Snarks}.

\begin{table}[htb!]
		\centering % used for centering table
		 \renewcommand{\arraystretch}{1.1} %1 standard
%		\begin{tabular}{cccc} % centered columns (4 columns)
		\begin{tabular}{crrr} % centered columns (4 columns)
			\toprule %inserts double horizontal lines
			Order & Girth 4  & Girth 5 & Total  \\ %[0.5ex] % inserts table
			%heading
			\midrule % inserts single horizontal line
44 & 0 & 31 & 31 \\
46 & $\geq$\,0 & $\geq$\,484 & $\geq$\,484 \\
48 & $\geq$\,1 112 & $\geq$\,4 793 & $\geq$\,5 905 \\
50 & $\geq$\,27 720 & $\geq$\,39 270 & $\geq$\,66 990 \\
52 & $\geq$\,457 285 & $\geq$\,356 488 & $\geq$\,813 773 \\
			%[1ex] % [1ex] adds vertical space
			\bottomrule %inserts single line
		\end{tabular}
	\caption{Counts of the smallest snarks of oddness 4 and
             cyclic connectivity 4.}\label{table:counts_snarks_oddness4}
	\end{table} 	

In \cite[Theorem 12]{LMMS} it was shown that if we allow
trivial snarks, the smallest one with oddness greater than 2
has 28 vertices and oddness 4. There are exactly three such
snarks, one with cyclic connectivity 3 and two with cyclic
connectivity 2, all three having girth $5$ (see~\cite{G} and
\cite[Corrigendum]{LMMS}). Since small snarks with oddness $4$
-- irrespectively of their connectivity -- can be useful for a
better understanding of oddness, we include
Table~\ref{table:counts_triv_snarks_oddness4} which lists the
counts of all trivial snarks of oddness 4 with girth
at least $4$ of orders from 28 to 34. The graphs from
Table~\ref{table:counts_triv_snarks_oddness4} were generated by
the first author in~\cite{G}. Restricting to girth at least $4$
is reasonable because every snark with oddness $4$ that
contains a triangle or a digon must arise from a smaller snark
with the same oddness in a straightforward manner
\cite[Lemma~1]{LMMS}.

%\begin{table}[htb!]
%		\centering % used for centering table
%		 \renewcommand{\arraystretch}{1.1} %1 standard
%%		\begin{tabular}{cccc} % centered columns (4 columns)
%		\begin{tabular}{crrr} % centered columns (4 columns)
%			\toprule %inserts double horizontal lines
%			Order & Girth 4  & Girth 5 & Total  \\ %[0.5ex] % inserts table
%			%heading
%			\midrule % inserts single horizontal line
%28 & 0 & 3 & 3 \\
%30 & 0 & 13 & 13 \\
%32 & 35 & 54 & 89 \\
%34 & 490 & 277 & 767 \\
%			%[1ex] % [1ex] adds vertical space
%			\bottomrule %inserts single line
%		\end{tabular} \caption{Counts of the smallest (trivial)
%snarks of oddness 4.} \label{table:counts_triv_snarks_oddness4}
%	\end{table}

\begin{table}[htb!]
		\centering % used for centering table
		\small
		 \renewcommand{\arraystretch}{1.1} %1 standard
		 %\setlength{\tabcolsep}{5pt} %6 standard
%		\begin{tabular}{cccc} % centered columns (4 columns)
		\begin{tabular}{c|rrr|rrrr} % centered columns (4 columns)
			\toprule %inserts double horizontal lines
			%Order & Girth 4  & Girth 5 & Total  \\ %[0.5ex] % inserts table
			%Order & Connectivity 2 & Connectivity 3 & Total & Connectivity 2 & Connectivity 3 & Total
            %\\ %[0.5ex] % inserts table			
		 \multirow{2}{*}{Order} & \multicolumn{3}{c}{Girth 4}  & \multicolumn{3}{c}{Girth 5}\\
		 %\cline{2-7}
		  %\cline{2-4} & \cline{5-7}
		  & Connectivity 2 & Connectivity 3 & Total & Connectivity 2 & Connectivity 3 & Total \\
		 %\hline			
			%heading
			\midrule % inserts single horizontal line
%28 & 0 & 3 & 3 \\
%30 & 0 & 13 & 13 \\
%32 & 35 & 54 & 89 \\
%34 & 490 & 277 & 767 \\
28 & 0 & 0 & 0 & 2 & 1 & 3 & \\
30 & 0 & 0 & 0 & 9 & 4 & 13 & \\
32 & 24 & 11 & 35 & 33 & 21 & 54 & \\
34 & 315 & 175 & 490 & 139 & 138 & 277 & \\
			%[1ex] % [1ex] adds vertical space
			\bottomrule %inserts single line
		\end{tabular}
	\caption{Counts of the smallest (trivial) snarks of oddness
4.}\label{table:counts_triv_snarks_oddness4}
	\end{table} 	

A glance at Table~\ref{table:counts_snarks_oddness4} reveals
that no snarks of oddness $4$ and girth $4$ on 46 vertices have
been found. A similar phenomenon occurs for the trivial snarks
in Table~\ref{table:counts_triv_snarks_oddness4}: there are no
snarks of oddness $4$ and girth $4$ of order 30. In contrast to
orders 44 and 28, we do not have any theoretical argument that
would exclude the existence of such snarks.

The graphs from Table~\ref{table:counts_snarks_oddness4} can be
downloaded from the \textit{House of Graphs}~\cite{BCGM} at
\url{http://hog.grinvin.org/Snarks} and the snarks of oddness 4
on 44 and 46 vertices from
Table~\ref{table:counts_snarks_oddness4} can be inspected at
the database of interesting graphs from the \textit{House of Graphs} by
searching for the keywords ``nontrivial snarks * oddness 4''.

\medskip

In the remainder of this section, we discuss several invariants
of the graphs from Tables~\ref{table:counts_snarks_oddness4}
and~\ref{table:counts_triv_snarks_oddness4}, again divided into
two groups -- uncolourability measures and general invariants.
All invariant values have been computed by two independent
programs or by programs which were already extensively tested
in earlier research. For example, the two independent programs
that we have used to compute the oddness of the graphs can be
downloaded from~\cite{oddness-site}. In the Appendix we
describe how some of the graphs from
Table~\ref{table:counts_snarks_oddness4} which appear
particularity interesting can be obtained.

\subsection{Resistance and other measures of uncolourability}

We begin our discussion with invariants which, in a certain
sense, measure uncolourability of cubic graphs. For all snarks
from Tables~\ref{table:counts_snarks_oddness4}
and~\ref{table:counts_triv_snarks_oddness4} we have determined
their resistance, perfect matching index, and the invariants
$\gamma_2$ and $\mu_3$. All of them have been defined earlier
in this paper and their values were discussed for the 31 snarks
in $\mathcal{M}$ in Section~\ref{sec:31}. We also know that all
snarks from Tables~\ref{table:counts_snarks_oddness4}
and~\ref{table:counts_triv_snarks_oddness4} have weak oddness
$4$ because this is true in general for all snarks  of oddness
$4$. The remaining invariants have been computed with the help
of a computer.

We start with the resistance of the snarks that we have
constructed.

\begin{observation}\label{obs:resistance4_52v}
All snarks from {\rm Table~\ref{table:counts_snarks_oddness4}}
and~{\rm \ref{table:counts_triv_snarks_oddness4}} have
resistance $3$, except for six snarks on $52$ vertices, which
have resistance $4$.
\end{observation}

The first example of a cyclically $4$-edge-connected snark with
resistance $4$ on 52 vertices was constructed by Lukot\!'ka et
al.\ \cite[Section~7]{LMMS}. The same snark is depicted and
discussed by Jin and Steffen in~\cite[Figure~3]{LiSt}. Two more
examples of order 52 can be easily obtained if we replace one
or both copies of $\mathbf{H}_1$ contained in this snark with
$\mathbf{H}_2$; all three of them are included among the
mentioned six snarks. Since we have no guarantee that these are
the smallest nontrivial snarks with resistance $4$, we offer
the following problem.

\begin{problem}\label{problem:smallest_resist4}
What is the smallest order of a
snark with resistance $4$? What is the smallest order of a
nontrivial snark with resistance $4$?
\end{problem}

A snark with oddness $4$ and connectivity $2$ on $40$
vertices was constructed in \cite[Section~6]{LMMS}. We have
verified that its resistance is $4$, so the upper bound for the
first question in Problem~\ref{problem:smallest_resist4} is $40$. On the other hand,
Observation~\ref{obs:resistance4_52v} implies that the lower
bound is $36$.

All nontrivial snarks with resistance $3$ known to us have
weak oddness $4$ (and hence oddness $4$ as well). It is
therefore tempting to ask the following questions.

\begin{problem} 
Does there exist a nontrivial
snark with resistance $3$ and weak oddness greater than $4$? Can the
difference be arbitrarily large?
\end{problem}

Allie~\cite{Al} constructed nontrivial snarks
demonstrating that the difference $\omega-\omega'$ can be
arbitrarily large. Unfortunately, all of them have resistance
$3$ and weak oddness $4$ while their oddness is arbitrarily
large.

\medskip

The remaining three uncolourability measures that we are going
to examine are related to the structure of perfect matchings in
snarks: the perfect matching index $\pi$ and the invariants
$\mu_3$ and $\gamma_2$ (see Section~\ref{sec:31}). The results
for perfect matching index are collected in the following
observation and~Table~\ref{table:counts_pmi_triv}.

\begin{observation} \label{obs:pmi}
All snarks of oddness $4$ from {\rm
Table~\ref{table:counts_snarks_oddness4}} have perfect matching
index $\pi = 4$, except for one graph of order $52$ which has
$\pi = 5$.
\end{observation}

	\begin{table}[htb!]
		\centering % used for centering table
		 \renewcommand{\arraystretch}{1.1} %1 standard
%		\begin{tabular}{cccc} % centered columns (4 columns)
		\begin{tabular}{crrr} % centered columns (4 columns)
			\toprule %inserts double horizontal lines
			Order & $\pi = 4$ & $\pi = 5$  & Total\\ %[0.5ex] % inserts table
			%heading
			\midrule % inserts single horizontal line
28 &  & 3 & 3 \\
30 &  & 13 & 13 \\
32 & 6 & 83 & 89 \\
34 & 40 & 727 & 767 \\
			%[1ex] % [1ex] adds vertical space
			\bottomrule %inserts single line
		\end{tabular}
		\caption{Counts of the value of perfect matching index
for the snarks from
Table~\ref{table:counts_triv_snarks_oddness4}.}
        \label{table:counts_pmi_triv}
	\end{table}		

While the previous observation seems to suggest that among
snarks with oddness at least $4$ those with perfect matching
index $5$ (or more) are rare, Table~\ref{table:counts_pmi_triv}
draws a different picture: only 46 among the 872
trivial snarks of oddness $4$ of orders 28--34 listed in
Table~\ref{table:counts_triv_snarks_oddness4} have perfect
matching index 4.

%The nontrivial snark with $\pi = 5$ from
%Observation~\ref{obs:pmi} can be inspected at the database of
%interesting graphs from the \textit{House of
%Graphs}~\cite{BCGM} by searching for the keywords ``snark *
%perfect matching index 5'' \textbf{(TODO)}.

%Table~\ref{table:counts_pmi_triv} shows the counts of the value
%of $\pi$ for the snarks of oddness~4 from
%Table~\ref{table:counts_triv_snarks_oddness4}.

\medskip

We proceed to the invariant $\mu_3$.
Tables~\ref{table:counts_mu3} and~\ref{table:counts_mu3_triv}
show the counts of the value of $\mu_3$ of the snarks of
oddness~4 from Tables~\ref{table:counts_snarks_oddness4}
and~\ref{table:counts_triv_snarks_oddness4}, respectively.
Again, there is a striking difference between the trivial and
the nontrivial snarks. Most cyclically 4-edge-connected snarks in
Table~\ref{table:counts_snarks_oddness4} have $\mu_3=6$, which
is the minimal possible value for snarks with $\omega=4$,
according to \cite[Corollary~2.4]{LiSt}. By contrast, most
trivial snarks from
Table~\ref{table:counts_triv_snarks_oddness4} have $\mu_3>6$.

\begin{table}[htb!]
		\centering % used for centering table
		 \renewcommand{\arraystretch}{1.1} %1 standard
%		\begin{tabular}{cccc} % centered columns (4 columns)
		\begin{tabular}{crrrrr} % centered columns (4 columns)
			\toprule %inserts double horizontal lines
			Order & $\mu_3 = 6$ & $\mu_3 = 7$ & $\mu_3 = 8$ & $\mu_3 = 9$ & Total\\ %[0.5ex] % inserts table
			%heading
			\midrule % inserts single horizontal line
44 & 31 &  &  & & 31 \\
46 & 481 & 1 & 2 & & 484 \\
48 & 5 878 & 1 & 26 & & 5 905 \\
50 & 66 724  & 5 & 261 & & 66 990 \\
52 & 809 349 & 2 213 & 2 181 & 30 & 813 773 \\
			%[1ex] % [1ex] adds vertical space
			\bottomrule %inserts single line
		\end{tabular}
		\caption{Counts of the value of $\mu_3$ for the snarks
from Table~\ref{table:counts_snarks_oddness4}. }
        \label{table:counts_mu3}
	\end{table}

	\begin{table}[htb!]
		\centering % used for centering table
		 \renewcommand{\arraystretch}{1.1} %1 standard
%		\begin{tabular}{cccc} % centered columns (4 columns)
		\begin{tabular}{crrrrr} % centered columns (4 columns)
			\toprule %inserts double horizontal lines
			Order & $\mu_3 = 6$ & $\mu_3 = 7$ & $\mu_3 = 8$ & $\mu_3 = 9$ & Total\\ %[0.5ex] % inserts table
			%heading
			\midrule % inserts single horizontal line
28 & 2 & 1 &  &  & 3 \\
30 & 2 & 4 & 4 & 3 & 13 \\
32 & 18 & 21 & 27 & 23 & 89 \\
34 & 118 & 145 & 239 & 265 & 767 \\
			%[1ex] % [1ex] adds vertical space
			\bottomrule %inserts single line
		\end{tabular}
		\caption{Counts of the value of $\mu_3$ for the snarks
from Table~\ref{table:counts_triv_snarks_oddness4}.}
        \label{table:counts_mu3_triv}
	\end{table}

Finally, we discuss the invariant $\gamma_2$.

\begin{observation}\label{obs:gamma2}
All snarks from {\rm Table~\ref{table:counts_snarks_oddness4}}
have $\gamma_2 = 2$, except for $30$ graphs of order $52$ which
have $\gamma_2 = 3$.
\end{observation}

The $30$ snarks with $\gamma_2 = 3$ from
Observation~\ref{obs:gamma2} are the same graphs as the $30$
snarks with $\mu_3=9$ from Table~\ref{table:counts_mu3}.

	\begin{table}[htb!]
		\centering % used for centering table
		 \renewcommand{\arraystretch}{1.1} %1 standard
%		\begin{tabular}{cccc} % centered columns (4 columns)
		\begin{tabular}{crrr} % centered columns (4 columns)
			\toprule %inserts double horizontal lines
			Order & $\gamma_2 = 2$ & $\gamma_2 = 3$  & Total\\ %[0.5ex] % inserts table
			%heading
			\midrule % inserts single horizontal line
28 & 3 &  & 3 \\
30 & 10 & 3 & 13 \\
32 & 66 & 23 & 89 \\
34 & 522 & 245 & 767 \\
			%[1ex] % [1ex] adds vertical space
			\bottomrule %inserts single line
		\end{tabular}
		\caption{Counts of the value of $\gamma_2$ for the
snarks from Table~\ref{table:counts_triv_snarks_oddness4}.}
        \label{table:counts_gamma2_triv}
	\end{table}		

Table~\ref{table:counts_gamma2_triv} shows the counts of the
value of $\gamma_2$ of the snarks of oddness~4 from
Table~\ref{table:counts_triv_snarks_oddness4}. Again, the
picture for trivial snarks is quite different. Most cyclically 4-edge-connected
snarks from Table~\ref{table:counts_snarks_oddness4} have
$\gamma_2=2$, which is the minimum possible value for snarks
with $\omega=4$, by \cite[Proposition~2.1]{Steffen:5-flows}. By
contrast, for trivial snarks the distribution of values
$\gamma_2=2$ and $\gamma_2=3$ is more even.

\subsection{Circumference}

Every hamiltonian cubic graph is $3$-edge-colourable, which
means that the circumference of every snark can be at most
$n-1$, where $n$ denotes the number of vertices. The
\emph{circumference deficit}, denoted here by $\xi$, is the
difference between order and circumference. It is not difficult
to observe that the circumference deficit of a cubic graph $G$
is at least its resistance $\rho(G)$, if the circumference
deficit is even, and at least $\rho(G)-1$, otherwise. In this
sense, circumference deficit can also be considered as one of
the measures of uncolourability of cubic graphs.

The well known dominating cycle conjecture~\cite{F84} implies
that every cyclically 4-edge-connected snark has circumference
at least $4n/3$, where $n$ is the order of the graph. On the
other hand, in~\cite{MM} M\'a\v cajov\'a and Maz\'ak
constructed a family of cyclically 4-edge-connected snarks on
$8m$ vertices with circumference $7m+2$. They also made a
conjecture that every cyclically 4-edge-connected cubic has
circumference at least $7n/8$.

Brinkmann et al.\ determined in~\cite{BrGHM} that nearly all
nontrivial snarks up to 36 vertices have circumference $n-1$
and the remainder (about 0.002\,\%) has circumference $n-2$. We
have determined the circumference of our snarks of oddness 4,
and the result is as follows.
%0.0020054092 percent has circ n-2

\begin{observation}\label{obs:circ}
All snarks of oddness $4$ from {\rm
Table~\ref{table:counts_snarks_oddness4}} have circumference
deficit  $3$, except for nine graphs on $52$ vertices which
have circumference deficit $4$.
\end{observation}

The nine graphs with circumference deficit $\xi = 4$ from
Observation~\ref{obs:circ} include the six graphs of
resistance~4 from Observation~\ref{obs:resistance4_52v}.

As can be seen from Table~\ref{table:counts_circ_triv}, the
behaviour of circumference on trivial snarks is again quite
different from that of the nontrivial ones.

\begin{table}[htb!]
		\centering % used for centering table
		 \renewcommand{\arraystretch}{1.1} %1 standard
%		\begin{tabular}{cccc} % centered columns (4 columns)
		\begin{tabular}{crrrr} % centered columns (4 columns)
			\toprule %inserts double horizontal lines
			Order &  $\xi=3$ & $\xi=4$ & $\xi=12$ & Total\\ %[0.5ex] % inserts table
			%heading
			\midrule % inserts single horizontal line
              28  &  3   &   &    & 3   \\
              30  &  13  &   &    & 13  \\
              32  &  88  &   &  1 & 89  \\
              34  &  760 & 5 &  2 & 767 \\
			%[1ex] % [1ex] adds vertical space
			\bottomrule %inserts single line
		\end{tabular}
		\caption{Counts of the circumference deficit of the
snarks from Table~\ref{table:counts_triv_snarks_oddness4}. }
\label{table:counts_circ_triv}
	\end{table}

\subsection{Automorphism group}

Tables~\ref{table:counts_snarks_oddness4_groupsize}
and~\ref{table:counts_triv_snarks_oddness4_groupsize} show
statistics of the number of automorphisms of the snarks from
Tables~\ref{table:counts_snarks_oddness4}
and~\ref{table:counts_triv_snarks_oddness4}, respectively. With
the exception of two snarks on 50 vertices, the automorphism
group of each of these graphs has order a power of $2$, three
times a power of $2$, or nine times a power of $2$, powers of
$2$ significantly prevailing.

	\begin{table}[htb!]
		\centering % used for centering table
		\small
		 \renewcommand{\arraystretch}{1.1} %1 standard
%		\begin{tabular}{cccc} % centered columns (4 columns)
		\begin{tabular}{crrrrrrrrrrr} % centered columns (4 columns)
			\toprule %inserts double horizontal lines
Order & 1  & 2 & 4 & 8 & 10 & 16 & 32 & 64 & 128 & 256 & Total\\ %[0.5ex] % inserts table
			%heading
			\midrule % inserts single horizontal line
44 & 8 & 8 & 9 & 3 &  & 2 &  & 1 &  &  & 31 \\
46 & 190 & 105 & 136 & 40 &  & 9 & 3 & 1 &  &  & 484 \\
48 & 2 480 & 1 361 & 1 602 & 353 &  & 93 & 10 & 5 & 1 &  & 5 905 \\
50 & 29 809 & 15 088 & 18 032 & 3 367 & 2 & 585 & 87 & 20 &  &  & 66 990 \\
52 & 382 025 & 171 584 & 221 499 & 33 586 &  & 4 350 & 629 & 84 & 12 & 4 & 813 773 \\
			%[1ex] % [1ex] adds vertical space
			\bottomrule %inserts single line
		\end{tabular}
		\caption{Counts of the order of the automorphism group
                 of the snarks %of oddness 4
                 from Table~\ref{table:counts_snarks_oddness4}.
\label{table:counts_snarks_oddness4_groupsize}}
	\end{table}

	\begin{table}[htb!]
		\centering % used for centering table
		\small
		 \renewcommand{\arraystretch}{1.1} %1 standard
%		\begin{tabular}{cccc} % centered columns (4 columns)
		\begin{tabular}{crrrrrrrrrrrrrrr} % centered columns (4 columns)
			\toprule %inserts double horizontal lines
Order & 4  & 8 & 16 & 24 & 32 & 48 & 64 & 96 & 128 & 192 & 256 & 288 & 384 & 768 & Total\\
			%heading
			\midrule % inserts single horizontal line
28 &  &  &  &  &  & 1 & 1 &  & 1 &  &  &  &  &  & 3  \\
30 &  & 1 & 2 &  & 3 &  & 5 &  &  &  &  & 1 & 1 &  & 13  \\
32 &  & 12 & 26 &  & 29 & 1 & 14 & 1 & 5 &  &  &  &  & 1 & 89  \\
34 & 6 & 198 & 267 & 1 & 183 & 3 & 81 & 2 & 21 & 1 & 3 &  & 1 &  & 767  \\
			%[1ex] % [1ex] adds vertical space
			\bottomrule %inserts single line
		\end{tabular}
		\caption{Counts of the order of the automorphism group
                 of the snarks %of oddness 4
                 from Table~\ref{table:counts_triv_snarks_oddness4}.
\label{table:counts_triv_snarks_oddness4_groupsize}}
	\end{table}	

\subsection{Genus}

Tables~\ref{table:counts_snarks_genus_oddness}
and~\ref{table:counts_snarks_genus_oddness_triv} show the genus
of the snarks of oddness 4 from
Tables~\ref{table:counts_snarks_oddness4}
and~\ref{table:counts_triv_snarks_oddness4} and reveal that
among the constructed snarks of oddness $4$ there are no snarks
of genus smaller than their resistance and no nontrivial snarks
of genus smaller than their oddness. On the other hand, it is
known that there exist infinitely many snarks of genus $1$
(Vodopivec~\cite{V08}) and, in fact, infinitely many snarks of
any given genus $g \geq 1$ (Mohar and Vodopivec \cite[Theorem~2.1]{MV}).
It is therefore tempting to ask the following two questions.

\medskip

%	\begin{table}[htb!]
	\begin{table}[htb!]
		\centering % used for centering table
		 \renewcommand{\arraystretch}{1.1} %1 standard
%		\begin{tabular}{cccc} % centered columns (4 columns)
		\begin{tabular}{crrrr} % centered columns (4 columns)
			\toprule %inserts double horizontal lines
			Order & Genus 4  & Genus 5 & Genus 6  & Total\\ %[0.5ex] % inserts table
			%heading
			\midrule % inserts single horizontal line
44 & 13 & 18 &  & 31 \\
46 & 42 & 442 &  & 484 \\
48 & 150 & 5 713 & 42 & 5 905 \\
50 & 531 & 61 642 & 4 817 & 66 990 \\
52 & 2 767 & 595 528 & 215 478 & 813 773 \\
			%[1ex] % [1ex] adds vertical space
			\bottomrule %inserts single line
		\end{tabular}
		\caption{Counts of the genus of the snarks from
Table~\ref{table:counts_snarks_oddness4}.}
                 \label{table:counts_snarks_genus_oddness}
	\end{table}

%	\begin{table}[htb!]
	\begin{table}[htb!]
		\centering % used for centering table
		 \renewcommand{\arraystretch}{1.1} %1 standard
%		\begin{tabular}{cccc} % centered columns (4 columns)
		\begin{tabular}{crrr} % centered columns (4 columns)
			\toprule %inserts double horizontal lines
			Order & Genus 3  & Genus 4   & Total\\ %[0.5ex] % inserts table
			%heading
			\midrule % inserts single horizontal line
28 & 3 &  & 3  \\
30 & 11 & 2 & 13  \\
32 & 54 & 35 & 89  \\
34 & 283 & 484 & 767  \\
			%[1ex] % [1ex] adds vertical space
			\bottomrule %inserts single line
		\end{tabular}
		\caption{Counts of the genus of the snarks from
Table~\ref{table:counts_triv_snarks_oddness4}.}
                 \label{table:counts_snarks_genus_oddness_triv}
	\end{table}

\begin{problem}\label{problem:genus_omega}
Does there exist a nontrivial
snark of oddness $\omega\ge 4$ with genus smaller than
$\omega$?
\end{problem}

\begin{problem}\label{problem:genus_rho}
Does there exist a snark of
resistance $\rho\ge 3$ with genus smaller than $\rho$?
\end{problem}

For $\rho=2$, the questions posed in Problems~\ref{problem:genus_omega} and~\ref{problem:genus_rho} have a
positive answer, see~\cite{MV, V08}.

%\medskip
%
%\jg{From Martin's email from 20 December 2018: I tend to
%believe that using these trivial snarks of genus 3 we could
%construct nontrivial snarks with oddness 4 and genus 3.}

\section{Final remarks}
\label{sect:final_remarks}

It is important to emphasise that Theorem~\ref{thm:44}, which
is the main result of our previous paper~\cite{GMS1}, does not
yet determine the order of a smallest nontrivial snark with
oddness at least 4. The reason is that it does not exclude the
existence of cyclically 5-edge-connected snarks with oddness at
least 4 on fewer than 44 vertices. However, the smallest
currently known cyclically 5-edge-connected snark with oddness
at least 4 has 76 vertices (see \cite[Section~8]{LMMS}), which
indicates that a cyclically 5-edge-connected snark with oddness
at least 4 on fewer than 44 vertices either does not exist or
will be very difficult to find. Thus the following problem
remains open.

\begin{problem}
Determine the smallest order of a
nontrivial snark with oddness at least $4$.
\end{problem}

If we take into account the fact that the validity of the cycle
double cover conjecture is open for snarks of oddness 6
\cite{HG}, the following problem appears to be interesting.

\begin{problem}\label{problem:smallest_nontriv_oddness6}
Determine the smallest order of a
nontrivial snark with oddness at least $6$.
\end{problem}

The problem remains interesting even in the version where
snarks are required only to be $2$-connected rather than
nontrivial, since trivial snarks with large oddness can serve
as ingredients for constructions of nontrivial ones. The best
current upper bound for Problem~\ref{problem:smallest_nontriv_oddness6} is 70 in the nontrivial
version, see \cite[Section~7]{LMMS} and 40 in the trivial
version~\cite[Section~6]{LMMS}. It follows from~\cite{G} that
the current lower bound for the trivial version is 36.

\subsection*{Acknowledgements}

We would like to thank Gunnar Brinkmann for providing us with
an independent program for computing the genus of a graph. Most
of the computations were carried out using the Stevin
Supercomputer Infrastructure at Ghent University.

%%%%%%%%%%%%%%%%%%%%%%%%%%%%%%%%%%%%%%%%%%%%%%%%%%%%%%%
% You do not have to use the same format for your references, but
%    include everything in this file.  Don't use natbib please.
% If you use BibTeX to create a bibliography, copy the .bbl file into here.

%\jg{TODO: don't forget to remove references which are not used.}

\section*{Appendix 1: Obtaining graphs from the Observations in
Section~\ref{sect:properties}}

Exceptional graphs mentioned in Observations~\ref{obs:resistance4_52v}-\ref{obs:circ} in
Section~\ref{sect:properties} can be obtained from the database of interesting
graphs in the \textit{House of Graphs}~\cite{BCGM} by searching keywords ``snark
* resistance 4", ``snark * perfect matching index 5", ``snark *
gamma", and ``snark * circumference deficit 4'', respectively.

%Here we describe how interesting graphs from the Observations
%in Section~\ref{sect:properties} can be obtained from the
%database of interesting graphs from the \textit{House of
%Graphs}~\cite{BCGM} by searching for certain keywords. More
%specifically:
%
%\begin{itemize}
%\item The six snarks with resistance 4 from
%    Observation~\ref{obs:resistance4_52v} can be found by
%    searching for the keywords ``snark * resistance 4''.
%
%\item The snark with perfect matching index 5 from
%    Observation~\ref{obs:pmi} can be found by searching for
%    the keywords ``snark * perfect matching index 5''.
%    \jg{Or better not to make this graph public as we will
%    probably use it in a future project? Cf.\ concern of
%    Edita...}
%
%\itemThe 30 snarks with $\gamma_2 = 3$ from
%    Observation~\ref{obs:gamma2} can be found by searching
%    for the keywords ``snark * gamma''. These are the same
%    graphs as the 30 snarks with $\mu_3=9$ from
%    Table~\ref{table:counts_mu3}.
%
%\item The nine graphs with circumference deficit $\xi = 4$
%    from Observation~\ref{obs:circ} can be found by
%    searching for the keywords ``snark * circumference
%    deficit 4''. These nine graphs include the six graphs
%    of resistance~4 from
%    Observation~\ref{obs:resistance4_52v}.
%
%%\item The trivial snarks of oddness 4 with circumference deficit
%%$\xi > 3$ from Table~\ref{table:counts_circ_triv} can be
%%found by searching for the keywords ``snark * circumference
%%deficit 4''.
%\end{itemize}
%
\section*{Appendix 2: Adjacency lists of the 31 snarks in $\mathcal{M}$}

We present the adjacency lists of the 31 snarks in
$\mathcal{M}$. These graphs can be downloaded at
\url{http://hog.grinvin.org/Snarks} or can be inspected at the
database of interesting graphs from the \textit{House of
Graphs}~\cite{BCGM} by searching for the keywords ``nontrivial
snarks * oddness 4''. The graphs are listed in the order in
which they were generated.

\begin{enumerate}
\addtolength{\itemsep}{-2mm}
\small

\item \{0: 12 14 27; 1: 6 9 16; 2: 4 9 17; 3: 5 7 39; 4: 2
    5 13; 5: 3 4 16; 6: 1 7 15; 7: 3 6 17; 8: 11 12 15; 9:
    1 2 35; 10: 11 13 14; 11: 8 10 34; 12: 0 8 13; 13: 4 10
    12; 14: 0 10 15; 15: 6 8 14; 16: 1 5 17; 17: 2 7 16;
    18: 22 26 29; 19: 31 36 38; 20: 25 27 42; 21: 23 25 43;
    22: 18 23 42; 23: 21 22 40; 24: 25 32 33; 25: 20 21 24;
    26: 18 27 43; 27: 0 20 26; 28: 29 32 34; 29: 18 28 30;
    30: 29 31 33; 31: 19 30 32; 32: 24 28 31; 33: 24 30 34;
    34: 11 28 33; 35: 9 37 38; 36: 19 37 39; 37: 35 36 40;
    38: 19 35 41; 39: 3 36 41; 40: 23 37 41; 41: 38 39 40;
    42: 20 22 43; 43: 21 26 42\} \quad (Class 1b)

\item \{0: 12 14 38; 1: 6 9 16; 2: 4 9 17; 3: 5 7 22; 4: 2
    5 13; 5: 3 4 16; 6: 1 7 15; 7: 3 6 17; 8: 11 12 15; 9:
    1 2 26; 10: 11 13 14; 11: 8 10 32; 12: 0 8 13; 13: 4 10
    12; 14: 0 10 15; 15: 6 8 14; 16: 1 5 17; 17: 2 7 16;
    18: 21 25 28; 19: 21 24 26; 20: 22 24 25; 21: 18 19 22;
    22: 3 20 21; 23: 24 31 42; 24: 19 20 23; 25: 18 20 26;
    26: 9 19 25; 27: 31 35 43; 28: 18 29 43; 29: 28 34 40;
    30: 31 36 39; 31: 23 27 30; 32: 11 35 37; 33: 37 39 40;
    34: 29 35 42; 35: 27 32 34; 36: 30 37 41; 37: 32 33 36;
    38: 0 39 41; 39: 30 33 38; 40: 29 33 41; 41: 36 38 40;
    42: 23 34 43; 43: 27 28 42\} \quad (Class 4a)

\item \{0: 12 14 28; 1: 6 9 16; 2: 4 9 17; 3: 5 7 24; 4: 2
    5 13; 5: 3 4 16; 6: 1 7 15; 7: 3 6 17; 8: 11 12 15; 9:
    1 2 18; 10: 11 13 14; 11: 8 10 36; 12: 0 8 13; 13: 4 10
    12; 14: 0 10 15; 15: 6 8 14; 16: 1 5 17; 17: 2 7 16;
    18: 9 22 25; 19: 32 40 42; 20: 22 24 26; 21: 23 24 25;
    22: 18 20 23; 23: 21 22 40; 24: 3 20 21; 25: 18 21 26;
    26: 20 25 27; 27: 26 29 42; 28: 0 33 34; 29: 27 38 41;
    30: 31 32 37; 31: 30 35 36; 32: 19 30 33; 33: 28 32 43;
    34: 28 35 37; 35: 31 34 38; 36: 11 31 39; 37: 30 34 39;
    38: 29 35 39; 39: 36 37 38; 40: 19 23 41; 41: 29 40 43;
    42: 19 27 43; 43: 33 41 42\} \quad (Class 4a)

\item \{0: 12 14 34; 1: 5 6 9; 2: 4 7 9; 3: 5 7 37; 4: 2 5
    13; 5: 1 3 4; 6: 1 7 8; 7: 2 3 6; 8: 6 11 15; 9: 1 2
    18; 10: 13 14 16; 11: 8 16 32; 12: 0 13 17; 13: 4 10
    12; 14: 0 10 15; 15: 8 14 17; 16: 10 11 17; 17: 12 15
    16; 18: 9 22 26; 19: 28 30 42; 20: 22 25 27; 21: 23 25
    26; 22: 18 20 23; 23: 21 22 38; 24: 25 29 31; 25: 20 21
    24; 26: 18 21 27; 27: 20 26 28; 28: 19 27 40; 29: 24 32
    35; 30: 19 33 35; 31: 24 33 34; 32: 11 29 33; 33: 30 31
    32; 34: 0 31 35; 35: 29 30 34; 36: 38 41 42; 37: 3 39
    41; 38: 23 36 39; 39: 37 38 43; 40: 28 41 43; 41: 36 37
    40; 42: 19 36 43; 43: 39 40 42\} \quad (Class 1b)

\item \{0: 4 8 12; 1: 20 22 32; 2: 7 9 24; 3: 5 7 25; 4: 0
    5 24; 5: 3 4 21; 6: 7 15 16; 7: 2 3 6; 8: 0 9 25; 9: 2
    8 23; 10: 12 15 17; 11: 19 20 23; 12: 0 10 13; 13: 12
    14 16; 14: 13 15 39; 15: 6 10 14; 16: 6 13 17; 17: 10
    16 26; 18: 19 21 22; 19: 11 18 37; 20: 1 11 21; 21: 5
    18 20; 22: 1 18 23; 23: 9 11 22; 24: 2 4 25; 25: 3 8
    24; 26: 17 29 38; 27: 31 32 34; 28: 30 33 34; 29: 26 31
    33; 30: 28 31 37; 31: 27 29 30; 32: 1 27 33; 33: 28 29
    32; 34: 27 28 36; 35: 37 40 42; 36: 34 41 42; 37: 19 30
    35; 38: 26 40 43; 39: 14 41 43; 40: 35 38 41; 41: 36 39
    40; 42: 35 36 43; 43: 38 39 42\} \quad (Class 3a)

\item \{0: 4 8 32; 1: 14 20 22; 2: 4 7 9; 3: 5 7 8; 4: 0 2
    5; 5: 3 4 21; 6: 15 16 26; 7: 2 3 37; 8: 0 3 9; 9: 2 8
    11; 10: 12 15 17; 11: 9 19 23; 12: 10 13 39; 13: 12 14
    16; 14: 1 13 15; 15: 6 10 14; 16: 6 13 17; 17: 10 16
    19; 18: 21 22 24; 19: 11 17 24; 20: 1 21 25; 21: 5 18
    20; 22: 1 18 23; 23: 11 22 25; 24: 18 19 25; 25: 20 23
    24; 26: 6 29 38; 27: 31 32 34; 28: 30 33 34; 29: 26 31
    33; 30: 28 31 37; 31: 27 29 30; 32: 0 27 33; 33: 28 29
    32; 34: 27 28 36; 35: 37 40 42; 36: 34 41 42; 37: 7 30
    35; 38: 26 40 43; 39: 12 41 43; 40: 35 38 41; 41: 36 39
    40; 42: 35 36 43; 43: 38 39 42\} \quad (Class 3a)

\item \{0: 12 14 41; 1: 5 6 9; 2: 4 7 9; 3: 5 7 30; 4: 2 5
    13; 5: 1 3 4; 6: 1 7 8; 7: 2 3 6; 8: 6 11 15; 9: 1 2
    19; 10: 13 14 16; 11: 8 16 43; 12: 0 13 17; 13: 4 10
    12; 14: 0 10 15; 15: 8 14 17; 16: 10 11 17; 17: 12 15
    16; 18: 22 26 30; 19: 9 23 31; 20: 22 25 27; 21: 23 25
    26; 22: 18 20 23; 23: 19 21 22; 24: 25 33 35; 25: 20 21
    24; 26: 18 21 27; 27: 20 26 38; 28: 29 33 36; 29: 28 40
    42; 30: 3 18 32; 31: 19 34 35; 32: 30 34 36; 33: 24 28
    34; 34: 31 32 33; 35: 24 31 36; 36: 28 32 35; 37: 38 40
    43; 38: 27 37 39; 39: 38 41 42; 40: 29 37 41; 41: 0 39
    40; 42: 29 39 43; 43: 11 37 42\} \quad (Class 3b)

\item \{0: 12 14 40; 1: 5 6 9; 2: 4 7 9; 3: 5 7 21; 4: 2 5
    13; 5: 1 3 4; 6: 1 7 8; 7: 2 3 6; 8: 6 11 15; 9: 1 2
    26; 10: 13 14 16; 11: 8 16 38; 12: 0 13 17; 13: 4 10
    12; 14: 0 10 15; 15: 8 14 17; 16: 10 11 17; 17: 12 15
    16; 18: 25 29 33; 19: 28 30 31; 20: 22 34 42; 21: 3 22
    43; 22: 20 21 31; 23: 24 30 32; 24: 23 35 36; 25: 18 26
    42; 26: 9 25 28; 27: 28 29 32; 28: 19 26 27; 29: 18 27
    30; 30: 19 23 29; 31: 19 22 32; 32: 23 27 31; 33: 18 34
    43; 34: 20 33 37; 35: 24 38 41; 36: 24 39 40; 37: 34 39
    41; 38: 11 35 39; 39: 36 37 38; 40: 0 36 41; 41: 35 37
    40; 42: 20 25 43; 43: 21 33 42\} \quad (Class 4b)

\item \{0: 12 14 29; 1: 6 9 16; 2: 4 9 17; 3: 5 7 35; 4: 2
    5 13; 5: 3 4 16; 6: 1 7 15; 7: 3 6 17; 8: 11 12 15; 9:
    1 2 40; 10: 11 13 14; 11: 8 10 28; 12: 0 8 13; 13: 4 10
    12; 14: 0 10 15; 15: 6 8 14; 16: 1 5 17; 17: 2 7 16;
    18: 22 26 29; 19: 31 37 38; 20: 22 25 27; 21: 23 25 26;
    22: 18 20 23; 23: 21 22 39; 24: 25 30 32; 25: 20 21 24;
    26: 18 21 27; 27: 20 26 28; 28: 11 27 42; 29: 0 18 33;
    30: 24 31 36; 31: 19 30 34; 32: 24 34 35; 33: 29 34 36;
    34: 31 32 33; 35: 3 32 36; 36: 30 33 35; 37: 19 39 43;
    38: 19 41 42; 39: 23 37 41; 40: 9 41 43; 41: 38 39 40;
    42: 28 38 43; 43: 37 40 42\} \quad (Class 3c)

\item \{0: 12 14 28; 1: 5 6 9; 2: 4 7 9; 3: 5 7 18; 4: 2 5
    13; 5: 1 3 4; 6: 1 7 8; 7: 2 3 6; 8: 6 11 15; 9: 1 2
    24; 10: 13 14 16; 11: 8 16 33; 12: 0 13 17; 13: 4 10
    12; 14: 0 10 15; 15: 8 14 17; 16: 10 11 17; 17: 12 15
    16; 18: 3 22 26; 19: 30 34 38; 20: 22 25 27; 21: 23 25
    26; 22: 18 20 23; 23: 21 22 34; 24: 9 25 42; 25: 20 21
    24; 26: 18 21 27; 27: 20 26 29; 28: 0 31 32; 29: 27 36
    39; 30: 19 31 33; 31: 28 30 42; 32: 28 37 43; 33: 11 30
    43; 34: 19 23 40; 35: 37 39 40; 36: 29 37 41; 37: 32 35
    36; 38: 19 39 41; 39: 29 35 38; 40: 34 35 41; 41: 36 38
    40; 42: 24 31 43; 43: 32 33 42\} \quad (Class 2)

\item \{0: 12 14 33; 1: 5 6 9; 2: 4 7 9; 3: 5 7 24; 4: 2 5
    13; 5: 1 3 4; 6: 1 7 8; 7: 2 3 6; 8: 6 11 15; 9: 1 2
    18; 10: 13 14 16; 11: 8 16 28; 12: 0 13 17; 13: 4 10
    12; 14: 0 10 15; 15: 8 14 17; 16: 10 11 17; 17: 12 15
    16; 18: 9 22 26; 19: 30 34 38; 20: 22 25 27; 21: 23 25
    26; 22: 18 20 23; 23: 21 22 34; 24: 3 25 42; 25: 20 21
    24; 26: 18 21 27; 27: 20 26 29; 28: 11 31 32; 29: 27 36
    39; 30: 19 31 33; 31: 28 30 42; 32: 28 37 43; 33: 0 30
    43; 34: 19 23 40; 35: 37 39 40; 36: 29 37 41; 37: 32 35
    36; 38: 19 39 41; 39: 29 35 38; 40: 34 35 41; 41: 36 38
    40; 42: 24 31 43; 43: 32 33 42\} \quad (Class 2)

\item \{0: 12 14 38; 1: 5 6 9; 2: 4 7 9; 3: 5 7 25; 4: 2 5
    13; 5: 1 3 4; 6: 1 7 8; 7: 2 3 6; 8: 6 11 15; 9: 1 2
    26; 10: 13 14 16; 11: 8 16 41; 12: 0 13 17; 13: 4 10
    12; 14: 0 10 15; 15: 8 14 17; 16: 10 11 17; 17: 12 15
    16; 18: 26 30 34; 19: 29 31 32; 20: 36 38 40; 21: 23 35
    42; 22: 23 36 43; 23: 21 22 32; 24: 25 31 33; 25: 3 24
    35; 26: 9 18 42; 27: 29 39 40; 28: 29 30 33; 29: 19 27
    28; 30: 18 28 31; 31: 19 24 30; 32: 19 23 33; 33: 24 28
    32; 34: 18 35 43; 35: 21 25 34; 36: 20 22 37; 37: 36 39
    41; 38: 0 20 39; 39: 27 37 38; 40: 20 27 41; 41: 11 37
    40; 42: 21 26 43; 43: 22 34 42\} \quad (Class 4b)

\item \{0: 3 14 22; 1: 5 6 9; 2: 4 7 9; 3: 0 5 7; 4: 2 5
    13; 5: 1 3 4; 6: 1 7 8; 7: 2 3 6; 8: 6 11 15; 9: 1 2
    26; 10: 13 14 16; 11: 8 16 29; 12: 13 17 30; 13: 4 10
    12; 14: 0 10 15; 15: 8 14 17; 16: 10 11 17; 17: 12 15
    16; 18: 21 25 38; 19: 21 24 26; 20: 22 24 25; 21: 18 19
    22; 22: 0 20 21; 23: 24 37 40; 24: 19 20 23; 25: 18 20
    26; 26: 9 19 25; 27: 31 39 41; 28: 32 34 37; 29: 11 33
    34; 30: 12 32 35; 31: 27 33 35; 32: 28 30 33; 33: 29 31
    32; 34: 28 29 35; 35: 30 31 34; 36: 38 41 42; 37: 23 28
    42; 38: 18 36 39; 39: 27 38 43; 40: 23 41 43; 41: 27 36
    40; 42: 36 37 43; 43: 39 40 42\} \quad (Class 4a)

\item \{0: 3 12 22; 1: 5 6 9; 2: 4 7 9; 3: 0 5 7; 4: 2 5
    13; 5: 1 3 4; 6: 1 7 8; 7: 2 3 6; 8: 6 11 15; 9: 1 2
    26; 10: 13 14 16; 11: 8 16 29; 12: 0 13 17; 13: 4 10
    12; 14: 10 15 30; 15: 8 14 17; 16: 10 11 17; 17: 12 15
    16; 18: 21 25 38; 19: 21 24 26; 20: 22 24 25; 21: 18 19
    22; 22: 0 20 21; 23: 24 37 40; 24: 19 20 23; 25: 18 20
    26; 26: 9 19 25; 27: 31 39 41; 28: 32 34 37; 29: 11 33
    34; 30: 14 32 35; 31: 27 33 35; 32: 28 30 33; 33: 29 31
    32; 34: 28 29 35; 35: 30 31 34; 36: 38 41 42; 37: 23 28
    42; 38: 18 36 39; 39: 27 38 43; 40: 23 41 43; 41: 27 36
    40; 42: 36 37 43; 43: 39 40 42\} \quad (Class 4a)

\item \{0: 12 14 22; 1: 5 6 9; 2: 4 7 9; 3: 5 7 30; 4: 2 5
    13; 5: 1 3 4; 6: 1 7 8; 7: 2 3 6; 8: 6 11 15; 9: 1 2
    29; 10: 13 14 16; 11: 8 16 26; 12: 0 13 17; 13: 4 10
    12; 14: 0 10 15; 15: 8 14 17; 16: 10 11 17; 17: 12 15
    16; 18: 21 25 38; 19: 21 24 26; 20: 22 24 25; 21: 18 19
    22; 22: 0 20 21; 23: 24 37 40; 24: 19 20 23; 25: 18 20
    26; 26: 11 19 25; 27: 31 39 41; 28: 32 34 37; 29: 9 33
    34; 30: 3 32 35; 31: 27 33 35; 32: 28 30 33; 33: 29 31
    32; 34: 28 29 35; 35: 30 31 34; 36: 38 41 42; 37: 23 28
    42; 38: 18 36 39; 39: 27 38 43; 40: 23 41 43; 41: 27 36
    40; 42: 36 37 43; 43: 39 40 42\} \quad (Class 1a)

\item \{0: 12 14 27; 1: 6 9 16; 2: 4 9 17; 3: 5 7 38; 4: 2
    5 13; 5: 3 4 16; 6: 1 7 15; 7: 3 6 17; 8: 11 12 15; 9:
    1 2 41; 10: 11 13 14; 11: 8 10 35; 12: 0 8 13; 13: 4 10
    12; 14: 0 10 15; 15: 6 8 14; 16: 1 5 17; 17: 2 7 16;
    18: 26 30 34; 19: 29 31 32; 20: 25 38 40; 21: 23 35 42;
    22: 23 36 43; 23: 21 22 32; 24: 25 31 33; 25: 20 24 37;
    26: 18 27 42; 27: 0 26 29; 28: 29 30 33; 29: 19 27 28;
    30: 18 28 31; 31: 19 24 30; 32: 19 23 33; 33: 24 28 32;
    34: 18 35 43; 35: 11 21 34; 36: 22 39 40; 37: 25 39 41;
    38: 3 20 39; 39: 36 37 38; 40: 20 36 41; 41: 9 37 40;
    42: 21 26 43; 43: 22 34 42\} \quad (Class 4b)

\item \{0: 12 14 38; 1: 6 9 16; 2: 4 9 17; 3: 5 7 30; 4: 2
    5 13; 5: 3 4 16; 6: 1 7 15; 7: 3 6 17; 8: 11 12 15; 9:
    1 2 33; 10: 11 13 14; 11: 8 10 35; 12: 0 8 13; 13: 4 10
    12; 14: 0 10 15; 15: 6 8 14; 16: 1 5 17; 17: 2 7 16;
    18: 21 25 28; 19: 24 26 42; 20: 22 24 43; 21: 18 22 42;
    22: 20 21 34; 23: 24 31 32; 24: 19 20 23; 25: 18 26 43;
    26: 19 25 40; 27: 28 31 33; 28: 18 27 29; 29: 28 30 32;
    30: 3 29 31; 31: 23 27 30; 32: 23 29 33; 33: 9 27 32;
    34: 22 36 39; 35: 11 37 39; 36: 34 37 38; 37: 35 36 40;
    38: 0 36 41; 39: 34 35 41; 40: 26 37 41; 41: 38 39 40;
    42: 19 21 43; 43: 20 25 42\} \quad (Class 1a)

\item \{0: 12 14 38; 1: 6 9 16; 2: 4 9 17; 3: 5 7 30; 4: 2
    5 13; 5: 3 4 16; 6: 1 7 15; 7: 3 6 17; 8: 11 12 15; 9:
    1 2 33; 10: 11 13 14; 11: 8 10 35; 12: 0 8 13; 13: 4 10
    12; 14: 0 10 15; 15: 6 8 14; 16: 1 5 17; 17: 2 7 16;
    18: 21 25 28; 19: 21 24 26; 20: 22 24 25; 21: 18 19 22;
    22: 20 21 34; 23: 24 31 32; 24: 19 20 23; 25: 18 20 26;
    26: 19 25 37; 27: 28 31 33; 28: 18 27 29; 29: 28 30 32;
    30: 3 29 31; 31: 23 27 30; 32: 23 29 33; 33: 9 27 32;
    34: 22 36 39; 35: 11 37 42; 36: 34 38 43; 37: 26 35 40;
    38: 0 36 41; 39: 34 41 42; 40: 37 41 43; 41: 38 39 40;
    42: 35 39 43; 43: 36 40 42\} \quad (Class 1a)

\item \{0: 12 14 29; 1: 6 9 16; 2: 4 9 17; 3: 5 7 38; 4: 2
    5 13; 5: 3 4 16; 6: 1 7 15; 7: 3 6 17; 8: 11 12 15; 9:
    1 2 36; 10: 11 13 14; 11: 8 10 23; 12: 0 8 13; 13: 4 10
    12; 14: 0 10 15; 15: 6 8 14; 16: 1 5 17; 17: 2 7 16;
    18: 22 26 30; 19: 32 36 40; 20: 22 25 27; 21: 23 25 26;
    22: 18 20 23; 23: 11 21 22; 24: 25 33 34; 25: 20 21 24;
    26: 18 21 27; 27: 20 26 29; 28: 30 35 42; 29: 0 27 41;
    30: 18 28 31; 31: 30 32 43; 32: 19 31 33; 33: 24 32 42;
    34: 24 35 43; 35: 28 34 39; 36: 9 19 37; 37: 36 39 41;
    38: 3 39 40; 39: 35 37 38; 40: 19 38 41; 41: 29 37 40;
    42: 28 33 43; 43: 31 34 42\} \quad (Class 2)

\item \{0: 12 14 40; 1: 5 6 9; 2: 4 7 9; 3: 5 7 21; 4: 2 5
    13; 5: 1 3 4; 6: 1 7 8; 7: 2 3 6; 8: 6 11 15; 9: 1 2
    23; 10: 13 14 16; 11: 8 16 38; 12: 0 13 17; 13: 4 10
    12; 14: 0 10 15; 15: 8 14 17; 16: 10 11 17; 17: 12 15
    16; 18: 24 28 32; 19: 27 29 30; 20: 22 33 42; 21: 3 22
    43; 22: 20 21 30; 23: 9 29 31; 24: 18 25 42; 25: 24 27
    34; 26: 27 28 31; 27: 19 25 26; 28: 18 26 29; 29: 19 23
    28; 30: 19 22 31; 31: 23 26 30; 32: 18 33 43; 33: 20 32
    37; 34: 25 35 36; 35: 34 38 41; 36: 34 39 40; 37: 33 39
    41; 38: 11 35 39; 39: 36 37 38; 40: 0 36 41; 41: 35 37
    40; 42: 20 24 43; 43: 21 32 42\} \quad (Class 4b)

\item \{0: 12 14 39; 1: 5 6 9; 2: 4 7 9; 3: 5 7 27; 4: 2 5
    13; 5: 1 3 4; 6: 1 7 8; 7: 2 3 6; 8: 6 11 15; 9: 1 2
    34; 10: 13 14 16; 11: 8 16 35; 12: 0 13 17; 13: 4 10
    12; 14: 0 10 15; 15: 8 14 17; 16: 10 11 17; 17: 12 15
    16; 18: 22 26 29; 19: 31 36 38; 20: 25 27 42; 21: 23 25
    43; 22: 18 23 42; 23: 21 22 40; 24: 25 32 33; 25: 20 21
    24; 26: 18 27 43; 27: 3 20 26; 28: 29 32 34; 29: 18 28
    30; 30: 29 31 33; 31: 19 30 32; 32: 24 28 31; 33: 24 30
    34; 34: 9 28 33; 35: 11 37 38; 36: 19 37 39; 37: 35 36
    40; 38: 19 35 41; 39: 0 36 41; 40: 23 37 41; 41: 38 39
    40; 42: 20 22 43; 43: 21 26 42\} \quad (Class 1b)

\item \{0: 4 8 12; 1: 20 22 32; 2: 7 9 24; 3: 5 7 25; 4: 0
    5 24; 5: 3 4 21; 6: 7 15 16; 7: 2 3 6; 8: 0 9 25; 9: 2
    8 23; 10: 12 15 17; 11: 19 20 23; 12: 0 10 13; 13: 12
    14 16; 14: 13 15 29; 15: 6 10 14; 16: 6 13 17; 17: 10
    16 35; 18: 19 21 22; 19: 11 18 33; 20: 1 11 21; 21: 5
    18 20; 22: 1 18 23; 23: 9 11 22; 24: 2 4 25; 25: 3 8
    24; 26: 33 36 40; 27: 35 37 38; 28: 30 41 42; 29: 14 30
    43; 30: 28 29 38; 31: 32 37 39; 32: 1 31 41; 33: 19 26
    42; 34: 35 36 39; 35: 17 27 34; 36: 26 34 37; 37: 27 31
    36; 38: 27 30 39; 39: 31 34 38; 40: 26 41 43; 41: 28 32
    40; 42: 28 33 43; 43: 29 40 42\} \quad (Class 4b)

\item \{0: 4 8 12; 1: 5 14 38; 2: 4 7 9; 3: 5 7 8; 4: 0 2
    5; 5: 1 3 4; 6: 7 16 18; 7: 2 3 6; 8: 0 3 9; 9: 2 8 22;
    10: 11 16 19; 11: 10 13 21; 12: 0 15 41; 13: 11 23 24;
    14: 1 17 18; 15: 12 17 19; 16: 6 10 17; 17: 14 15 16;
    18: 6 14 19; 19: 10 15 18; 20: 21 23 30; 21: 11 20 25;
    22: 9 23 25; 23: 13 20 22; 24: 13 25 34; 25: 21 22 24;
    26: 29 33 36; 27: 32 34 42; 28: 30 32 43; 29: 26 30 42;
    30: 20 28 29; 31: 32 39 40; 32: 27 28 31; 33: 26 34 43;
    34: 24 27 33; 35: 36 39 41; 36: 26 35 37; 37: 36 38 40;
    38: 1 37 39; 39: 31 35 38; 40: 31 37 41; 41: 12 35 40;
    42: 27 29 43; 43: 28 33 42\} \quad (Class 3b)

\item \{0: 12 14 37; 1: 6 9 16; 2: 4 9 17; 3: 5 7 34; 4: 2
    5 13; 5: 3 4 16; 6: 1 7 15; 7: 3 6 17; 8: 11 12 15; 9:
    1 2 32; 10: 11 13 14; 11: 8 10 18; 12: 0 8 13; 13: 4 10
    12; 14: 0 10 15; 15: 6 8 14; 16: 1 5 17; 17: 2 7 16;
    18: 11 22 26; 19: 28 30 42; 20: 22 25 27; 21: 23 25 26;
    22: 18 20 23; 23: 21 22 38; 24: 25 29 31; 25: 20 21 24;
    26: 18 21 27; 27: 20 26 28; 28: 19 27 40; 29: 24 32 35;
    30: 19 33 35; 31: 24 33 34; 32: 9 29 33; 33: 30 31 32;
    34: 3 31 35; 35: 29 30 34; 36: 38 41 42; 37: 0 39 41;
    38: 23 36 39; 39: 37 38 43; 40: 28 41 43; 41: 36 37 40;
    42: 19 36 43; 43: 39 40 42\} \quad (Class 1b)

\item \{0: 8 12 16; 1: 11 13 14; 2: 7 18 22; 3: 5 17 24; 4:
    5 25 38; 5: 3 4 14; 6: 7 13 15; 7: 2 6 21; 8: 0 9 24;
    9: 8 11 41; 10: 11 12 15; 11: 1 9 10; 12: 0 10 13; 13:
    1 6 12; 14: 1 5 15; 15: 6 10 14; 16: 0 17 25; 17: 3 16
    23; 18: 2 19 34; 19: 18 21 23; 20: 21 22 30; 21: 7 19
    20; 22: 2 20 23; 23: 17 19 22; 24: 3 8 25; 25: 4 16 24;
    26: 29 33 36; 27: 32 34 42; 28: 30 32 43; 29: 26 30 42;
    30: 20 28 29; 31: 32 39 40; 32: 27 28 31; 33: 26 34 43;
    34: 18 27 33; 35: 36 39 41; 36: 26 35 37; 37: 36 38 40;
    38: 4 37 39; 39: 31 35 38; 40: 31 37 41; 41: 9 35 40;
    42: 27 29 43; 43: 28 33 42\} \quad (Class 4b)

\item \{0: 12 14 35; 1: 5 6 9; 2: 4 7 9; 3: 5 7 29; 4: 2 5
    13; 5: 1 3 4; 6: 1 7 8; 7: 2 3 6; 8: 6 11 15; 9: 1 2
    28; 10: 13 14 16; 11: 8 16 40; 12: 0 13 17; 13: 4 10
    12; 14: 0 10 15; 15: 8 14 17; 16: 10 11 17; 17: 12 15
    16; 18: 22 26 29; 19: 31 37 38; 20: 22 25 27; 21: 23 25
    26; 22: 18 20 23; 23: 21 22 39; 24: 25 30 32; 25: 20 21
    24; 26: 18 21 27; 27: 20 26 28; 28: 9 27 42; 29: 3 18
    33; 30: 24 31 36; 31: 19 30 34; 32: 24 34 35; 33: 29 34
    36; 34: 31 32 33; 35: 0 32 36; 36: 30 33 35; 37: 19 39
    43; 38: 19 41 42; 39: 23 37 41; 40: 11 41 43; 41: 38 39
    40; 42: 28 38 43; 43: 37 40 42\} \quad (Class 3c)

\item \{0: 12 14 18; 1: 6 9 16; 2: 4 9 17; 3: 5 7 28; 4: 2
    5 13; 5: 3 4 16; 6: 1 7 15; 7: 3 6 17; 8: 11 12 15; 9:
    1 2 33; 10: 11 13 14; 11: 8 10 24; 12: 0 8 13; 13: 4 10
    12; 14: 0 10 15; 15: 6 8 14; 16: 1 5 17; 17: 2 7 16;
    18: 0 22 26; 19: 30 34 38; 20: 22 25 27; 21: 23 25 26;
    22: 18 20 23; 23: 21 22 34; 24: 11 25 42; 25: 20 21 24;
    26: 18 21 27; 27: 20 26 29; 28: 3 31 32; 29: 27 36 39;
    30: 19 31 33; 31: 28 30 42; 32: 28 37 43; 33: 9 30 43;
    34: 19 23 40; 35: 37 39 40; 36: 29 37 41; 37: 32 35 36;
    38: 19 39 41; 39: 29 35 38; 40: 34 35 41; 41: 36 38 40;
    42: 24 31 43; 43: 32 33 42\} \quad (Class 2)

\item \{0: 1 2 3; 1: 0 40 41; 2: 0 39 43; 3: 0 38 42; 4: 9
    13 17; 5: 8 12 17; 6: 7 15 19; 7: 6 16 18; 8: 5 9 14;
    9: 4 8 11; 10: 13 16 22; 11: 9 12 20; 12: 5 11 26; 13:
    4 10 25; 14: 8 15 21; 15: 6 14 23; 16: 7 10 24; 17: 4 5
    20; 18: 7 22 28; 19: 6 21 27; 20: 11 17 32; 21: 14 19
    37; 22: 10 18 36; 23: 15 27 37; 24: 16 28 36; 25: 13 30
    35; 26: 12 29 34; 27: 19 23 32; 28: 18 24 33; 29: 26 31
    35; 30: 25 31 34; 31: 29 30 33; 32: 20 27 40; 33: 28 31
    41; 34: 26 30 39; 35: 25 29 39; 36: 22 24 38; 37: 21 23
    38; 38: 3 36 37; 39: 2 34 35; 40: 1 32 43; 41: 1 33 42;
    42: 3 41 43; 43: 2 40 42\} \quad (Class 5a)

\item \{0: 1 2 3; 1: 0 40 41; 2: 0 39 43; 3: 0 38 42; 4: 9
    13 17; 5: 8 12 17; 6: 7 15 19; 7: 6 16 18; 8: 5 9 14;
    9: 4 8 11; 10: 13 16 22; 11: 9 12 20; 12: 5 11 26; 13:
    4 10 25; 14: 8 15 21; 15: 6 14 23; 16: 7 10 24; 17: 4 5
    20; 18: 7 22 28; 19: 6 21 27; 20: 11 17 32; 21: 14 19
    37; 22: 10 18 36; 23: 15 27 37; 24: 16 28 36; 25: 13 31
    35; 26: 12 29 33; 27: 19 23 32; 28: 18 24 34; 29: 26 30
    35; 30: 29 31 34; 31: 25 30 33; 32: 20 27 40; 33: 26 31
    41; 34: 28 30 39; 35: 25 29 39; 36: 22 24 38; 37: 21 23
    38; 38: 3 36 37; 39: 2 34 35; 40: 1 32 43; 41: 1 33 42;
    42: 3 41 43; 43: 2 40 42\} \quad (Class 6)

\item \{0: 1 2 3; 1: 0 38 39; 2: 0 41 43; 3: 0 40 42; 4: 5
    6 37; 5: 4 30 36; 6: 4 29 35; 7: 8 9 22; 8: 7 15 21; 9:
    7 14 20; 10: 14 18 31; 11: 15 19 32; 12: 14 17 34; 13:
    15 16 33; 14: 9 10 12; 15: 8 11 13; 16: 13 26 32; 17:
    12 25 31; 18: 10 25 34; 19: 11 26 33; 20: 9 21 24; 21:
    8 20 23; 22: 7 23 24; 23: 21 22 28; 24: 20 22 27; 25:
    17 18 28; 26: 16 19 27; 27: 24 26 30; 28: 23 25 29; 29:
    6 28 36; 30: 5 27 35; 31: 10 17 39; 32: 11 16 39; 33:
    13 19 43; 34: 12 18 42; 35: 6 30 38; 36: 5 29 38; 37: 4
    40 41; 38: 1 35 36; 39: 1 31 32; 40: 3 37 43; 41: 2 37
    42; 42: 3 34 41; 43: 2 33 40\} \quad (Class 5b)

\item \{0: 1 2 3; 1: 0 39 42; 2: 0 38 43; 3: 0 40 41; 4: 7
    10 18; 5: 8 9 15; 6: 9 12 16; 7: 4 13 17; 8: 5 10 12;
    9: 5 6 11; 10: 4 8 11; 11: 9 10 21; 12: 6 8 21; 13: 7
    14 30; 14: 13 19 22; 15: 5 20 24; 16: 6 26 29; 17: 7 22
    27; 18: 4 25 28; 19: 14 20 23; 20: 15 19 32; 21: 11 12
    35; 22: 14 17 33; 23: 19 31 35; 24: 15 31 36; 25: 18 29
    37; 26: 16 28 37; 27: 17 30 36; 28: 18 26 34; 29: 16 25
    34; 30: 13 27 33; 31: 23 24 32; 32: 20 31 38; 33: 22 30
    39; 34: 28 29 39; 35: 21 23 41; 36: 24 27 38; 37: 25 26
    40; 38: 2 32 36; 39: 1 33 34; 40: 3 37 43; 41: 3 35 42;
    42: 1 41 43; 43: 2 40 42\} \quad (Class 6)

\end{enumerate}

\end{document}